\documentclass[reqno]{amsart}
\usepackage{amsmath, amsthm, amssymb,  enumitem}

\usepackage[nobysame, non-sorted-cites, alphabetic]{amsrefs}
\usepackage{stmaryrd} 
\usepackage{hyperref}
\usepackage{verbatim} 
\usepackage{tikz}
    \usetikzlibrary{cd, graphs, decorations.pathmorphing, decorations.markings}
\usepackage{color}
\setenumerate{label=\textnormal{(\arabic*)}} 

\newcommand{\cat}[1]{\ensuremath{\mathsf{#1}}}

\newcommand{\C}{\mathcal{C}}
\newcommand{\catC}{\mathcal{C}}

\newcommand{\cC}{\ensuremath{\mathcal{C}}}

\DeclareMathOperator{\Ab}{\cat{Ab}}

\DeclareMathOperator{\Set}{\cat{Set}}

\DeclareMathOperator{\Grp}{\cat{Grp}}

\DeclareMathOperator{\HMag}{\cat{HMag}}

\DeclareMathOperator{\Msc}{\cat{Msc}}
\DeclareMathOperator{\cMsc}{\cat{cMsc}}
\DeclareMathOperator{\cSMsc}{\cat{cSMsc}}
\DeclareMathOperator{\SMsc}{\cat{SMsc}}
\DeclareMathOperator{\sMsc}{\cat{(S)Msc}}
\DeclareMathOperator{\Can}{\cat{Can}}
\DeclareMathOperator{\HGrp}{\cat{HGrp}}

\DeclareMathOperator{\End}{End}

\newcommand{\Z}{\mathbb{Z}}
\newcommand{\Q}{\mathbb{Q}}

\newcommand{\K}{\mathbf{K}}

\renewcommand{\P}{\mathcal{P}}

\newcommand{\ostar}{\circledast}

\newcommand{\modmod}{/\!/}

\newcommand{\separate}{\bigskip}

\numberwithin{equation}{section}

\theoremstyle{plain}
\newtheorem{theorem}[equation]{Theorem}
\newtheorem{corollary}[equation]{Corollary}
\newtheorem{lemma}[equation]{Lemma}
\newtheorem{proposition}[equation]{Proposition}
\newtheorem{question}[equation]{Question}

\theoremstyle{definition}
\newtheorem{definition}[equation]{Definition}

\newtheorem{example}[equation]{Example}
 
\newtheorem{remark}[equation]{Remark}

\begin{document}

\title[Quotient hyperstructures]{Quotients of mosaics and related hyperstructures}
\author{Siddhant Jajodia}
\email{jajodias@uci.edu}
\author{So Nakamura}
\email{snakamu2@uci.edu}
\author{Manuel Reyes}
\email{mreyes57@uci.edu}
\address{Department of Mathematics\\ University of California, Irvine \\
340 Rowland Hall\\ Irvine, CA 96267--3875 \\ USA}

\date{July 23, 2026}
\thanks{This work was supported by NSF grant DMS-2201273}

\begin{abstract}
This is a thorough study of quotients of hyperstructures that generalize hypergroups, namely mosaics and semimosaics. The quotients in these categories generalize those studied previously in the literature on hypergroups. We describe the effective congruences in these categories by characterizing them in terms of their underlying equivalence relation. This characterization is applied to provide new methods of constructing quotient objects modulo the action of endomorphisms, as well as to study explicit quotient mosaics of some small groups. 
We also show that the category of mosaics has a natural proto-exact structure.
\end{abstract}

\maketitle

\setcounter{tocdepth}{1} 
\makeatletter
\def\l@subsection{\@tocline{2}{0pt}{2.5pc}{5pc}{}}
\def\l@subsubsection{\@tocline{2}{0pt}{5pc}{7.5pc}{}} 
\makeatother

\section{Introduction}
\label{sec:intro}

For algebraic structures such as groups and modules, quotient structures are easy to describe: the kernel of every morphism is a subobject, and the quotient by this subobject is sufficient to describe all possible quotient structures.
The situation for rings is similar, although it is slightly complicated by the fact that ideals are not subobjects in the category of unital rings. However, when one studies monoids and other structures in universal algebra, the full complexity of quotient structures becomes evident. 
Rather than quotients by subobjects or ideals, we must allow for quotients by certain equivalence relations on the structures.

In this paper, we are interested in studying quotient objects of certain \emph{hyperstructures}, which lie outside the traditional bounds of universal algebra. These are sets $M$ equipped with binary operations of the form $\star \colon M \times M \to \P(M)$, so that products are \emph{sets} of elements in $M$. Two well-known types of hyperstructures are hypergroups~\cite{Marty} and hyperrings~\cite{Krasner}. Quotients of these structures have been studied quite extensively, beginning with the work of Marty~\cite{Marty:quotients}, and in many subsequent contexts, such as~\cite{Krasner:quotient, Jantosciak, FZ}.
However, only recently was a careful study of the categories of these objects undertaken, showing that the category of hypergroups fails to allow many types of universal constructions, including coequalizers.
Recently, two of the authors introduced in~\cite{NakamuraReyes} a generalization of hypergroups called~\emph{mosaics} whose category is quite well behaved, being both complete and cocomplete. The same is true for the corresponding structures that do not necessarily have inverse elements, which we call \emph{semimosaics}. (Definitions of these objects are reviewed in Section~\ref{sec:hyperstructures}.)

In this paper we take up the problem of understanding quotients of mosaics, which necessarily includes quotients of hypergroups that are more general than those previously considered in the literature.
We treat this problem from the perspective of congruences in regular categories, which we now recall. (We refer readers  to~\cite{Grillet:regular}, \cite[Chapter~2]{Borceux:handbook2}, \cite[Appendix~A5]{BorceuxBorn}, or~\cite{BournGran} for further details.)

\subsection{Reminder on quotients in regular categories}\label{subsection:regularcat}

Let $\C$ be a category with finite limits.
A \emph{relation}~\cite[Section~2.5]{Borceux:handbook2} on an object $M \in \C$ is a subobject of the product 
\begin{equation}\label{eq:relation}
r = (r_1, r_2) \colon R \hookrightarrow M \times M. 
\end{equation}
The relation is \emph{reflexive} if the diagonal $M \to M \times M$ factors through $r$, \emph{symmetric} if the ``swap'' automorphism $\sigma \colon M \times M \to M \times M$ satisfies $\sigma \circ r = r$, and \emph{transitive} if the pullback
\[
\begin{tikzcd}
R \times_M R \ar[r, dashrightarrow, "\rho_2"] \ar[d, dashrightarrow, "\rho_1"] & R \ar[d, "r_1"] \\
R \ar[r, "r_2"] & M
\end{tikzcd}
\]
is such that $(r_1 \circ \rho_1, r_2 \circ \rho_2) \colon R \times_M R \to M \times M$ factors through $r$.
Then $R$ is a \emph{congruence} 
on $M$ if it is reflexive, symmetric, and transitive. 
(An alternative characterization can be given as follows. If $R$ is a relation then for every object $X \in \C$ we obtain a relation $R_X$ on the set $\C(X,M)$ as the image of 
\[
\C(X,R) \hookrightarrow \C(X,M \times M) \cong \C(X,M) \times \C(X,M).
\]
Then, $R$ is a congruence if and only if each $R_X$ is an equivalence relation.)

Every morphism gives rise to a congruence in the following way.
If $f \colon M \to N$ is a morphism in $\C$, its \emph{kernel pair} $R_f$ is the pullback of $f$ along itself
\begin{equation}\label{eq:kernel pair}
\begin{tikzcd}
& &[-25pt] M \ar[dr, "f"] & \\[-1em]
R_f \ar[urr] \ar[drr] \ar[r, dashrightarrow] & M \times M \ar[ur] \ar[dr] \ar[rr, "f \times f"]  & & N. \\[-1em]
& & M \ar[ur, swap, "f"] & 
\end{tikzcd}
\end{equation}
The kernel pair $R_f$ can be viewed as a subobject of $M \times M$ and thus a relation on $M$. In fact, $R_f$ is a congruence, as shown in~\cite[Example~A.2.6]{BorceuxBorn}.
A congruence $R$ is called an \emph{effective congruence} if it is the kernel pair of some morphism in $\C$.

Given a congruence of the form~\eqref{eq:relation}, if the coequalizer 
\[
\begin{tikzcd}
R \ar[r, shift left, "r_1"] \ar[r, shift right, swap, "r_2"] & M \ar[r, "q"] & M/R,
\end{tikzcd}
\]
exists, then it is called the \emph{quotient} of $M$ by $R$.
In this situation, note that the relation $r = (r_1,r_2)$ factors uniquely through the kernel pair $\overline{R} := R_q$ as $R \hookrightarrow \overline{R} \hookrightarrow M \times M$, and their quotients coincide: $M/R = M/\overline{R}$.
As shown in the proof of~\cite[Proposition~A.5.12]{BorceuxBorn},
if $R$ is an effective congruence and its quotient $q \colon M \to M/R$ exists, then in fact $R$ is the kernel pair of its quotient: $R = \overline{R}$.

Recall that a \emph{regular epimorphism} in a category is a coequalizer of a parallel pair of morphisms. 
As explained in~\cite[Proposition~2.5.7]{Borceux:handbook1}, a morphism $f$ in $\C$ is a regular epimorphism (i.e., coequalizer) if and only if it is the coequalizer of its kernel pair $R_f \rightrightarrows M$. 
Thus regular epimorphisms whose domain is $M$ are described up to isomorphism as the quotients by all possible effective congruences on $M$ whose quotients exist.

\separate

A category category $\C$ is defined to be \emph{regular}~\cite{Grillet:regular} if it satisfies the conditions:
\begin{enumerate}[label=\textnormal{(\roman*)}] 
\item $\C$ has finite limits, 
\item pullbacks preserve regular epimorphisms, and
\item every morphism in $\C$ decomposes as a regular epimorphism followed by a monomorphism.
\end{enumerate}
The factorization condition~(iii) can be replaced~\cite[\S 2.1]{Borceux:handbook2} instead by the following condition on congruences:
\begin{enumerate}
    \item[\textnormal{(iii')}] every effective congruence has a quotient.
\end{enumerate}
It follows from~(iii') and the discussion above that for an object $M$ in a regular category $\catC$, we have bijective correspondences between regular epimorphic images of $M$ and effective congruences on $M$ in the following way:
\begin{equation}\label{eq:quotient correspondence}
\begin{tikzcd}[row sep=tiny]
\left\{\begin{array}{c}
    \textnormal{regular epimorphisms} \\
    f \colon M \twoheadrightarrow M'
    \end{array}\right\} \ar[r, leftrightarrow, "\sim"]
    & 
\left\{\begin{array}{c}
    \textnormal{effective congruences} \\
    R \hookrightarrow M \times M
    \end{array}\right\} \\
f \ar[r, mapsto] & R_f \\
(q \colon M \twoheadrightarrow M/R) & \ar[l, mapsto] R.
\end{tikzcd}
\end{equation}

Regular categories have a satisfactory theory of image factorization: every morphism has a (unique up to isomorphism) factorization as a regular epimorphism followed by a monomorphism, and these factorizations are pullback-stable. 
To be precise, if $f \colon M \to N$ is a morphism in a regular category $\C$, then this factorization is given by the factorization of $f$ through the quotient $q$ by its kernel pair $R_f$:
\begin{equation}\label{eq:epimono}
\begin{tikzcd}
M \ar[rr, "f"] \ar[dr, twoheadrightarrow, "q"]& & N \\
& M/R_f \ar[ur, hookrightarrow, "i"] &
\end{tikzcd}
\end{equation}
where $i$ is uniquely determined. Furthermore, regular epimorphisms are closed under composition~\cite[Corollary~2.23]{BournGran}.

Finally, a \emph{Barr exact} category is a regular category in which every congruence is effective. If $\C$ is Barr exact, then for every object $M$ in $\C$, ~\eqref{eq:quotient correspondence} simplifies further: the quotients of $M$ are in bijective correspondence with the isomorphism classes of congruences on $M$. 

\subsection{Quotients of hyperstructures}

It is well known that all varieties of algebras in the sense of universal algebra are Barr exact~\cite[Example~A.5.16]{BorceuxBorn}. In particular, groups, rings, modules over a ring, and monoids all form Barr exact categories. Thus quotients of an object $M$ of any of these categories are entirely characterized by congruences on that object, which are simply equivalence relations $R \subseteq M \times M$ such that $R$ is a subobject in the category.

It was shown in~\cite[Theorem~1.2]{NakamuraReyes} that the categories of mosaics and semimosaics are regular. 
However, as we will see in Example~\ref{ex:not exact}, these categories are \emph{not} Barr exact. That is, there exist congruences on mosaics which are not effective.
So in order to describe all quotients of a mosaic as in~\eqref{eq:quotient correspondence}, we must characterize which congruences are effective. We accomplish this goal in
Theorem~\ref{thm:effective congruence} and Corollary~\ref{cor:mosaic quotients}. We state the main results for the category of mosaics below, although the paper proves the corresponding statements for semimosaics as well.

Throughout this paper, we let ``congruence" refer to the (categorical) congruences defined above and reserve ``equivalence relation" for set-theoretic equivalence relations (i.e. congruences in $\Set$).

\begin{theorem}
Given a mosaic $M$, the isomorphism classes of effective congruences on $M$ bijectively correspond to the equivalence relations $\equiv$ on $M$ that satisfy the conditions:
\begin{enumerate}[label=\textnormal{(\roman*)}]
    \item if $x \in y' \star e' \cup e' \star y'$ for some $y' \equiv y$ and $e' \equiv e$, then $x \equiv y$
    \item if $x \equiv y$ then $x^{-1} \equiv y^{-1}$.
\end{enumerate}
The quotient of $M$ corresponding to such an equivalence relation is given by the set-theoretic quotient $q \colon M \to M/\!\!\equiv$, where the set of equivalence classes is equipped with the following hyperoperation:
\[
a \star b = q(q^{-1}(a) \star q^{-1}(b)).
\]  
\end{theorem}

\begin{proof}
    This summarizes Theorem~\ref{thm:effective congruence}, Corollary~\ref{cor:mosaic quotients} and Example~\ref{ex:not exact} for the category of mosaics.
\end{proof}

The fact that congruences on mosaics require the special conditions~(i) and~(ii) to be effective strongly suggests that $\Msc$ is not Barr exact, and we verify this with an example in Remnark~\ref{rem:malcev}. There are various weaker ``exactness'' conditions that one could ask for, and we show that some of these fail to hold for $\Msc$.

Notably, there is one ``exactness'' property that \emph{does} hold for the category of mosaics. This is the property of being a \emph{parabelian} category~\cite{Mozgovoy}, which is in turn defined as a pointed category (i.e., category with zero object) for which the normal mono- and epimorphisms form a \emph{proto-exact} structure~\cite{DK}. These definitions are reviewed in Subsection~\ref{sub:exact}. 

The following summarizes our results in this direction.

\begin{theorem}
The category $\Msc$ of mosaics is parabelian, but it is not Barr exact, protomodular, Malcev, or proto-abelian. 
\end{theorem}

\begin{proof}
This follows from Remark~\ref{rem:malcev}, Theorem~\ref{thm:parabelian}, and Remark~\ref{rem:proto-abelian}.
\end{proof}
\separate

The outline of this paper is as follows.
In Section~\ref{sec:hyperstructures}, we recall the definitions of hypermagmas, hypergroups, semimosaics, and mosaics, along with some features of their categories that were established in~\cite{NakamuraReyes}. 
Section~\ref{sec:congruences} focuses on the study of congruences and effective congruences in the categories of mosaics and semimosaics, concluding with a proof that the category of mosaics is parabelian. 
Finally, in Section~\ref{sec:examples} we apply our characterization to explore some examples of effective congruences on certain mosaics. This includes quotients by endomorphisms and groups of automorphisms, as well as some examples of quotient mosaics of some small groups.

\section{Background on Hyperstructures}
\label{sec:hyperstructures}

In this section, we provide definitions of the various objects studied in this paper, along with some previously established results concerning them. Much of the content is from \cite{NakamuraReyes}, to which we refer the reader for a more detailed discussion.

\begin{definition}
    A \textit{hypermagma} $(M, \star)$ is a set $M$ endowed with a function $$\star:M\times M\to \P(M).$$
\end{definition}

The function $\star$ above is called the \textit{hyperoperation} of $M$. The hyperoperation extends to a binary operation on $\P(M)$ by defining, for $X, Y\in \P(M)$:
$$
X\star Y = \bigcup_{(x, y) \in X\times Y} x\star y.
$$
By a slight abuse of notation, we will often identify an element $x\in M$ with the singleton $\{x\}\in \P(M)$. Then, we say that the hyperoperation $\star$ on $M$ (or simply, $M$) is \textit{associative} if for every $x, y, z\in M$,
$$
(x\star y)\star z = x\star(y \star z)
$$
where $\star$ is understood as the extended operation on $\P(M)$. Similarly, $\star$ (or $M$) is \textit{commutative} if for every $x, y\in M$,
$$
x\star y = y\star x.
$$

For a hypermagma $(M, \star)$, an element $e\in M$ is an \textit{identity} if for every $x\in M$, 
$$
x\star e = x = e\star x.
$$
It is easy to show that such an identity, if it exists, must be unique. This leads to the following definition.
\begin{definition}
    A \textit{semimosaic} $(M, \star, e)$ is a hypermagma with identity.
\end{definition} 

For a semimosaic $(M, \star, e)$, an \textit{inverse} of an element $x\in M$ is any element $x'\in M$ such that 
$$
e\in (x\star x')  \cap (x'\star x).
$$
Similarly, if $(M, \star)$ is a hypermagma, then $\star$ (or $M$) is called \textit{reversible} if there exists a function $(\cdot)^{-1}:M\to M$ such that for every $x, y, z\in M$,
$$
z\in x\star y \implies z\star y^{-1} \ni x \mbox{ and } x^{-1} \star z \ni y.
$$
A reversible operation thus allows for ``solving" membership relations for hypermagmas the way one might solve equations for groups. 

In general, an element of a semimosaic might have more than one inverse. However, if the semimosaic is reversible, then one can show that every element $x$ must have a unique inverse $x^{-1}$, namely the image of $x$ under the function $(\cdot)^{-1}$ above. This brings us to the following definition.
\begin{definition}
    A \textit{mosaic} $(M, \star, e, (\cdot)^{-1})$ is a reversible semimosaic. 
\end{definition}

Note that none of the structures so far are required to be associative. 

\begin{definition}
    A \textit{hypergroup} is an associative mosaic. A \textit{canonical hypergroup} is a commutative hypergroup.
\end{definition}

A hyperoperation on a hypermagma $M$ is called \textit{total} if its range is contained in $\P(M)\setminus \varnothing$. It follows from \cite[Lemma 2.6]{NakamuraReyes} that the hyperoperation on a hypergroup is always total, but this is not, in general, the case for mosaics or semimosaics. Finally, observe that a group is a hypergroup whose hyperoperation is single-valued.

(We remark that the terminology above, which agrees with that of~\cite{Zieschang}, differs from some of the mainstream literature on hypergroups \cite{CL, Jantosciak, MM}. What we call a hypergroup might otherwise be called a \emph{polygroup} or a \emph{reversible hypergroup with scalar identity}.)

A \textit{morphism of hypermagmas} is a function $f:M\to N$ such that for every $x, y \in M$,
$$
f(x\star y) \subseteq f(x)\star f(y).
$$
If we have equality in place of inclusion above, then the morphism is called \textit{strict}. A \emph{morphism of semimosaics} $f:M\to N$ is a morphism of hypermagmas that additionally satsifies $f(e_M)=e_N$, where $e_M$ and $e_N$ are the identities of $M$ and $N$ respectively. If $M$ and $N$ are mosaics, then any morphisms of semimosaics $f:M\to N$ also preserves inverses: for any $x\in M$, $e_N=f(e_M)\in f(x\star x^{-1})\subseteq f(x)\star f(x^{-1})$, and reversibility yields $f(x^{-1})=f(x)^{-1}$. We, therefore, define a morphism of mosaics to simply be a function between mosaics that is a morphism of semimosaics. A morphism of hypergroups too is defined in the same way. 

With morphisms defined so, hypermagmas, semimosaics, mosaics and hypergroups form categories that we will denote as $\HMag$, $\SMsc$, $\Msc$ and $\HGrp$ respectively. Then, letting $\Grp$ be the category of groups, we have the following chain of inclusions:
$$
\Grp \subseteq \HGrp \subseteq \Msc \subseteq \SMsc \hookrightarrow \HMag.
$$
Here, $\subseteq$ denotes a full subcategory and $\hookrightarrow$ denotes a faithful embedding. We also let $\cMsc$ and $\cSMsc$ denote the full subcategories of $\Msc$ and $\SMsc$ of commutative mosaics and commutative semimosaics respectively. 

It is shown in~\cite{NakamuraReyes} that the categories $\Msc$, $\SMsc$ and $\cMsc$ are complete, cocomplete, and regular, and that the forgetful functor from $\Msc$ to $\Set$ has a left adjoint (so ``free" objects exist in $\Msc$). Moreover, the one-element mosaic $\{e\}$ with $e\star e=e$ acts as a zero object in $\SMsc$ and $\Msc$. \cite{NakamuraReyes} also shows that neither $\HGrp$ nor $\Can$, the full subcategory of $\HGrp$ of canonical hypergroups, is complete or cocomplete.

\begin{example}
Hypergroups classically arise as quotients of ordinary groups. Suppose $G$ is a group and $\equiv$ is an equivalence relation on $G$ such that for every $a, b\in G$:
    \begin{itemize}
        \item The setwise product $[a][b] = \{a'b' : a' \equiv a \hspace{1mm} \text{and} \hspace{1mm} b'\equiv b$\} is a union of equivalence classes (where $[a]$ denotes the equivalence class of $a$),
        \item $[e][a] = [a] = [a][e]$, and
        \item The setwise inversion $[a]^{-1} = \{a'^{-1}: a' \equiv a\} = [a^{-1}]$.
        \end{itemize}
Then, the set-theoretic quotient $G/\!\!\equiv$ forms a hypergroup with 
$$
[a]\star [b] := \{[c] : [c]\subseteq [a][b]\}=\{[c]: c\in [a][b]\}.
$$
Moreover, the quotient map $G\twoheadrightarrow  G/\!\!\equiv$ is a morphism of hypergroups.

\begin{itemize}
    \item Suppose a group $G$ acts on a group $H$ by automorphisms. Then, the orbit equivalence relation $\equiv$ on $H$ where $a\equiv b$ iff $g(a)=b$ for some $g\in G$ satisfies the conditions above, turning the set of orbits $H/G$ into a hypergroup with hyperoperation as defined above. For instance, the multiplicative group $\{\pm 1\}$ acts on $\Z$ by multiplication, yielding the quotient hypergroup $\Z /\{\pm 1\}$, where $[1]+[1] = \{[0], [2]\}$.
    \item We let $\K$ denote the hypergroup whose underlying set is $\{0, 1\}$, where 0 acts as the identity and $1+1=\{0, 1\}$. Then, $\K = F/F^\times$ by the quotient construction above for any field $F$ with at least 3 elements and where $F^\times$ acts on $F$ by multiplication. This is the underlying hypergroup of the Krasner hyperfield \cite{Krasner}.
    \item   By the construction above, the conjugacy classes of a group $G$ form a hypergroup which is canonical even if $G$ is not abelian.
\end{itemize}
\end{example}

\begin{example}
        There is a wealth of examples of mosaics that are not hypergroups, in the form of matroids. A closure space $(M, C)$ is a set $M$ with a map $C:\P(M)\to \P(M)$ called the closure operator that satisfies, for every $S, T\subseteq M$, (i) $S\subseteq C(S)$, (ii) $S\subseteq T \implies C(S)\subseteq C(T)$, and (iii) $C(C(S)) = C(S)$. A \textit{matroid} is a closure space $(M, C)$ that also satisfies the exchange property: for every $x, y\in M$ and $S\subseteq M$, $x\notin C(S)$ and $x\in C(S\cup \{y\})$ implies $y\in C(S\cup \{x\}).$ A matroid $(M, C)$ is a \textit{simple pointed matroid} if there exists $0\in M$ such that $C(\varnothing) = \{0\}$ and $C(x) = \{0, x\}$ for every $x\in M$. Given a simple pointed matroid $(M, C, 0)$, we can turn it into a commutative mosaic with underlying set $M$ by defining, for every $x, y\in M$:

         \[ x + y = \begin{cases} 
            C(x, y)\setminus \{x, y, 0\} & x\neq y \\
            \{x, 0\} & x=y \\
        \end{cases}
        \]    
     In fact, \cite{NakamuraReyes} show that there is a faithful functor from the category of simple, pointed matroids (with morphisms appropriately defined) into $\cMsc$.
\end{example}

\subsection{Subobjects and quotient objects}

 In an abelian category such as $\Ab$, there are various ``flavors'' of monomorphisms and epimorphisms that always coincide but which must be carefully separated in other categories. This separation occurs in $\SMsc$ and $\Msc$ as well as in the respective full subcategories $\cSMsc$ and $\cMsc$. Thus, these categories are not abelian. \cite{NakamuraReyes} gives a complete characterization of the various monomorphisms and epimorphisms in these categories, which we recall in this section. We will be particularly interested in the difference between types of quotients that one can take, and therefore, in the characterizations of the epimorphisms.

First, we recall the following definitions for any category $\C$:
\begin{itemize}
\item A morphism in $\C$ is a \emph{regular} monomorphism (resp., epimorphism) if it is an equalizer (resp., coequalizer) of a parallel pair of morphisms in $\C$.
\item Suppose $\C$ has a zero object. A morphism in $\C$ is a \emph{normal} monomorphism (resp., epimorphism) if it is the kernel (resp., cokernel) of a morphism in $\C$.
\end{itemize}
This gives a hierarchy of monomorphisms and epimorphisms as follows:
\begin{center}
normal mono (epi) $\implies$ regular mono (epi) $\implies$ mono (epi).
\end{center}

The monomorphisms and epimorphisms in $\HMag$, $\Msc$ and $\SMsc$ are respectively the injective and surjective morphisms \cite{NakamuraReyes}. For hypermagmas $L\subseteq M$, if the inclusion $i:L\to M$ is a morphism of hypermagmas (and thus a monomorphism), then we call $L$ a \textit{subhypermagma} of $M$. We similarly define the notions of \textit{subsemimosaic} and \textit{submosaic}. Note that in each of these three categories, a monomorphism of the form $i \colon L \to M$ is \emph{not} completely described by the sub\textit{set} $i(L) \subseteq M$, because this does not uniquely determine the hyperoperation on $L$. Dually, an epimorphism $p \colon M \twoheadrightarrow N$ is not fully described by the equivalence relation corresponding to the set-theoretic quotient of $M$ by $p$ for the similar reason that this does not uniquely specify a hyperoperation on $N$. Thus, when specifying a subobject or quotient object in these categories, one must specify a subset or set-theoretic quotient along with the intended hyperoperation. Consider the following examples.

\begin{example}
Consider $\K =\{0, 1\}$ as in Example 2.5. Then, $\K$ and $\Z_2$ are both submosaics of $\K$ with the same underlying set but with distinct hyperoperations on them: $1+1=\{0, 1\}$ in $\K$ whereas $1+1=0$ in $\Z_2$. 

Similarly, consider the commutative mosaic $D$ with underlying set $\{0, 1, -1\}$ where $0$ acts as the identity, $1-1=0$, and all other sums are empty. Then, $p:D\to \K$ where $0\mapsto 0$ and $1, -1\mapsto 1$, and $p':D\to \Z_2$ where also $0\mapsto 0$ and $1, -1\mapsto 1$ are both surjective morphisms with codomains that arise from the same set-theoretic equivalence relation on $D$ but which possess distinct hyperoperations.
\end{example}

Fortunately, regular monomorphisms and epimorphisms do not suffer from the same shortcomings. We recall their characterization from~\cite[Theorem 1.2]{NakamuraReyes}. First, we have the following definition.

\begin{definition} In the categories $\SMsc$ and $\Msc$:
\begin{itemize}
    \item A morphism $p \colon M \to N$ is \emph{short} if it satisfies
 \[
 x \star y = p(p^{-1}(x) \star p^{-1}(y)) \quad \mbox{for all } x,y \in N.
 \]
    \item A morphism $i \colon L \to M$ is \emph{coshort} if it satisfies
 \[
 i^{-1}(i(x) \star i(y)) = x \star y \quad \mbox{for all } x,y \in L.
 \]
\end{itemize}
    
\end{definition}

It follows (setting $y = e$) that a short morphism is surjective, and that the hyperoperation of $N$ is uniquely determined by the hyperoperation of $M$ and the function $p$. Similarly, we see that a coshort morphism is injective and that the hyperoperation of $L$ is uniquely determined as above. 

Then in each of the categories $\SMsc$ and $\Msc$, we have the following equivalence of properties of morphisms:
\begin{center}
regular epimorphism $\iff$ short morphism, \\
regular monomorphism $\iff$ coshort morphism.
\end{center}
Given any subset $L$ of a hypermagma $(M, \star_M)$, there is a canonical hyperoperation that $L$ inherits from $M$ that is defined as follows:
\[
x\star_L y := (x\star_M y)\cap L.
\] 
This hyperoperation turns $L$ into a subhypermagma of $M$, and we call such an $L$ a \textit{regular subhypermagma} of $M$. Similarly, if $M$ is a semimosaic, then $L$ is a \textit{regular subsemimosaic} of $M$ if it is a regular subhypermagma of $M$ that contains the identity. If $M$ is a mosaic, then $L$ is a \textit{regular submosaic} of $M$ if it is a regular subsemimosaic of $M$ that is also closed under inversion.

It was observed in~\cite{NakamuraReyes} that a coshort morphism $i: L \to M$ in $\HMag$ is the same as an isomorphism of $L$ onto the regular subhypermagma with underlying set $i(L)$ of $M$. Thus, the coshort morphisms and consequently, the regular monomorphisms in $\HMag$, $\SMsc$ and $\Msc$ correspond to the regular subobjects in the above sense.\footnote{Note that regular subhypermagmas, and regular sub(semi)mosaics are called \emph{weak} subhypermagmas and weak sub(semi)mosaics in \cite{NakamuraReyes}. We use the term \textit{regular} because of the correspondence to regular monomorphisms.}

Finally, it was shown in~\cite{NakamuraReyes} that normal monomorphisms are precisely the strict injective morphisms. Another way to understand this is as follows. For a hypermagma $M$, we define a \emph{strict subhypermagma} of $M$ to be a subset $L \subseteq M$ that satisfies
\[
x,y \in L \implies x \star y \subseteq L.
\] 
Setting products in $L$ to be the same as those in $M$, then, $L$ inherits the structure of a hypermagma with the inclusion $L \hookrightarrow M$ being a strict morphism. Similarly, if $M$ is a semimosaic, then we say $L$ is a \emph{strict subsemimosaic} of $M$ if it is a strict subhypermagma of $M$ that contains the identity. If $M$ is a mosaic, then $L$ is a \emph{strict submosaic} of $M$ if it is a strict subsemimosaic of $M$ that is also closed under inversion. 

We say that a subset $L$ of a hypermagma $M$ is \emph{absorptive} if, for all $x \in M$,
\[
(x \star L \cup L\star x) \cap L \neq \varnothing \implies x \in L,
\] or in elementwise form,
\[
a \in x \star b \hspace{1mm} \text{or} \hspace{1mm} a\in b\star x \mbox{ for some } a,b \in L \implies x \in L.
\]
Note that if $M$ is a mosaic and $L$ is a strict submosaic, $L$ is automatically absorptive by reversibility. 
Then, normal monomorphisms in $\SMsc$ and $\Msc$ with codomain $M$ correspond precisely to strict absorptive subobjects in the above sense:
\begin{center}
\begin{tabular}{c}
normal monomorphism to \\
a (semi)mosaic $M$ \end{tabular}
$\iff$ strict (absorptive) subobject $L \subseteq M$.
\end{center}
Thus, we have so far, three levels of monomorphisms (ordinary, regular and normal) corresponding to three levels of morphisms (injective, coshort and strict) and three levels of subobjects (ordinary, regular and strict) respectively.

To understand normal epimorphisms in $\SMsc$ and $\Msc$, we explicitly describe the type of quotient object that arises. The case of regular epimorphisms is the subject of the next section. Of course, any cokernel morphism is the cokernel of its own kernel. So it suffices to describe normal epimorphisms that are coequalizers of the form
\[
\begin{tikzcd}
L \ar[r, shift left, "i"] \ar[r, shift right, swap, "0"] & M \ar[r, twoheadrightarrow] & N
\end{tikzcd}
\]
where $i$ is a normal monomorphism, or up to isomorphism, where $L$ is a strict subobject of $M$. In this case, we will denote the cokernel of the inclusion $L \hookrightarrow M$ by $M/L$. The case where $M$ is a canonical hypergroup was discussed in~\cite[Section~3]{Jun:geometry}.

\begin{lemma}\label{lem:cokernel}
Let $M$ be a mosaic (resp., semimosaic) with a strict submosaic (resp., strict absorptive subsemimosaic) $L \subseteq M$. The cokernel of the inclusion
\[
L \hookrightarrow M \overset{p}{\twoheadrightarrow} M/L
\]
has underlying set $M/L = M/\!\!\equiv$ given by the quotient of $M$ by the equivalence relation defined by $x \equiv y$ if and only if there is a sequence $x = z_1, z_2, \dots, z_n = y$ of elements in $M$ such that 
\[
  z_{i+1} \in z_i \star L \cup L\star z_i \hspace{3mm} \text{or} \hspace{3mm} z_{i} \in z_{i+1} \star L \cup L\star z_{i+1} \hspace{1mm}
\]
for $i = 1, \dots, n-1$. The hyperoperation on $M/L$ has the following form for $a,b \in M/L$: 
\[
a \ostar b = p(p^{-1}(a)\star p^{-1}(b)).
\]

If $M$ is a mosaic that is also associative (i.e. a hypergroup), then the equivalence relation takes the simpler form $x\equiv y$ if and only if $y\in L\star x\star L$. Thus, the underlying set of $M/L$ is the set of double cosets $\{L\star x\star L : x\in M\}$.
\end{lemma}

\begin{proof}
The first set of claims for $M$ a general mosaic or semimosaic follow from ~\cite[Corollary~4.2, Lemma~3.10]{NakamuraReyes}. 

Now, suppose $M$ is an associative mosaic i.e. a hypergroup. First, if $y\in L\star x\star L$, then $y\in l\star x\star l'$ for some $l, l' \in L$. This implies that $y\in c \star l'$ for some $c\in l\star x$. Thus, $x\equiv y$ by the definition of $\equiv$ above. Conversely, suppose that $x\equiv y$. Since $e\in L$, $z_{i+1} \in z_i \star L \cup L\star z_i$ implies that $z_{i+1} \in L\star z_i\star L$, and $z_{i} \in z_{i+1} \star L \cup L\star z_{i+1}$ implies that $z_i \in L \star z_{i+1} \star L$. Moreover, if $z_{i+1} \in L\star z_i \star L$, then $z_{i+1} \in l\star z_i \star l'$ for some $l, l' \in L$. Therefore, $z_{i+1} \in c \star l'$ for some $c\in l \star z_i$. By reversibility, $z_i \in l^{-1} \star c$ and $c\in z_{i+1} \star l'^{-1}$. So, $z_i \in l^{-1} \star z_{i+1} \star l'^{-1} \subseteq L\star z_{i+1} \star L$. Thus, we have shown that if $x\equiv y$, then there is a sequence $x = z_1, z_2, \dots, z_n = y$ of elements in $M$ such that $ z_{i+1} \in L\star z_i \star L$ for $i = 1, \dots, n-1$. Associativity and the fact that $L\star L\subseteq L$ then give us that $y\in L\star x\star L$. We conclude that $x\equiv y$ iff $y\in L\star x\star L$. This implies that the equivalence class of $x$ is the double coset $L\star x \star L$ containing $x$.
In this case, the hyperoperation on $M/L$ can thus be written as:
\begin{align*}
    (L\star x\star L) \ostar (L\star y \star L) &= \{L\star z \star L : z\in (L\star x\star L) \star (L\star y \star L)\} \\
    &= \{L\star z \star L : z\in L\star x\star L\star y \star L\}. \qedhere
\end{align*}
\end{proof}

\section{Effective congruences on mosaics and semimosaics}
\label{sec:congruences}

Because regular and normal epimorphisms do not coincide for (semi)mosaics, we require a more careful kind of quotient construction to describe regular epimorphisms. 
As discussed in Section~\ref{sec:intro}, this is achieved by considering \emph{congruences} on mosaics and semimosaics, similar to what is done in universal algebra. 
An important difference is that we must characterize \emph{effective} congruences among all congruences.

\subsection{Congruences}

We first characterize congruences in the categories of mosaics and semimosiacs. These are described in terms of the set-theoretic image of the subobjects.

\begin{lemma}\label{lem:congruences}
    Let $M$ be a (semi)mosaic. Then, the congruences on $M$ are precisely the sub(semi)mosaics of $M\times M$ whose underlying sets are equivalence relations on $M$.
\end{lemma}

\begin{proof}
     By Remark 2.5, the forgetful functors from $\SMsc$ and $\Msc$ to $\Set$ preserve products and pullbacks. Moreover, the monomorphisms in $\SMsc$ and $\Msc$ are precisely the injective morphisms. Then, one can check from the definition of congruence, that given a congruence $i: R \hookrightarrow M\times M$ on a (semi)mosaic $M$, its image $i(R)\subseteq M\times M$ is a set-theoretic equivalence relation on $M$, and given a sub(semi)mosaic $R \subseteq M\times M$ that is a set-theoretic equivalence relation on $M$, the inclusion map $i:R\hookrightarrow M\times M$ gives a congruence on $M$.
\end{proof}

\begin{remark}
    Every equivalence relation on a semimosaic is the underlying set of some congruence on it, but this is not true, in general, for a mosaic. If $R$ is an equivalence relation on a semimosaic $M$, reflexivity implies that the identity $(e_M, e_M) \in M \times M$ lies in $R$, so that $R$ can be made into a subobject of $M\times M$ by giving it, for instance, the regular subsemimosaic structure. This would make $R$ a congruence on $M$ in $\SMsc$ by Lemma~\ref{lem:congruences}. 
    
    But, if $M$ is a mosaic, then $R$ might not be closed under inverses as a subset of the mosaic $M\times M$ and hence, might not be the underlying set of any subobject of $M\times M$ in $\Msc$. 
For example, take the sign hypergroup $\mathbb{S}=\{0, \pm1\}$ and consider $R:=\Delta\cup\{(0, -1), (-1, 0)\}$ where $\Delta\subset\mathbb{S}\times\mathbb{S}$ is the diagonal set. Suppose that the subset $R$ has a mosaic structure such that the inclusion $R\hookrightarrow\mathbb{S}\times\mathbb{S}$ is a mosaic homomorphism. Then $(0,0) \in R$ is the identity element and $-(1, 1)=(-1, -1)$ since mosaic homomorphisms preserves identity elements and inverses. Thus we have $-(0, -1)=(0, 1)$. This contradicts that the inclusion is a mosaic homomorphism. 
\end{remark}

  Given a congruence \[
 i \colon R \hookrightarrow M \times M
 \]
 on a (semi)mosaic
$M$, the categorical quotient $p$ of $M$ by $R$ exists in $\sMsc$ by ~\cite[Theorem~3.11]{NakamuraReyes} but as described in the proof of that result, its underlying set may not coincide with the set-theoretic quotient of $M$ by $R$ but rather a further quotient hypermagma. We reproduce ~\cite[Example~3.6]{NakamuraReyes} that illustrates this.

\begin{example}
     Let $D$ be the commutative semimosaic on the set $\{0, 1, 2\}$ with $0$ as the identity, and with $1+1=1+2=2+2=\{0, 1, 2\}$. Let $F$ be the free commutative semimosaic on $\{0, 1, 2\}$ with 0 as the identity. Let $f,g:F\to D$ be semimosaic morphisms defined by $f(1)=1$, $f(2)=2$, and $g(1)=0$, $g(2)=2$. Then, the set-theoretic quotient of $f$ and $g$ is $E:=\{[0], [2]\}$. Let $p:D\to E$ be the quotient map. Suppose there was some hyperoperation on $E$ that made $E$ into a semimosaic and $p$ into a morphism of semimosaics. Then, 
    $$E = p(\{0, 1, 2\})=p(1+1)\subseteq p(1)+p(1) = [0]+[0] = [0]$$
    which is not possible. Thus, $E$ cannot be the underlying set of the quotient of $f$ and $g$ in $\SMsc$. Obtaining the quotient will involve further identifying $[0]$ and $[2]$ in $E$.
\end{example}

\subsection{Effective congruences}

The following definition captures the ``good'' property that characterizes when the the set-theoretic quotient forms a (semi)mosaic in a natural way.

\begin{definition}\label{def:mosaic congruences}
Let $M$ be a semimosaic. Let $\equiv$ be an equivalence relation on $M$, and for each $x \in M$ let $[x]$ denote its equivalence class under $\equiv$. We say that $\equiv$ is \emph{semimosaic equivalence} if it satisfies the following condition for all $x,y \in M$:
\[
y \in x \star [e] \mbox{ or }y \in [e] \star x \implies [y] = [x].
\]
If $M$ is a mosaic, we say that $\equiv$ is a \emph{mosaic equivalence} if it is a semimisoaic equivalence that satisfies $[x^{-1}]=[x]^{-1}$ for all $x \in M$.
\end{definition}

Note that if $M$ is a mosaic, then a semimosaic equivalence on $M$ need not be a mosaic equivalence. The following example illustrates this.

\begin{example}
    Consider the free mosaic $F_2=\{e, a^{\pm 1}, b^{\pm 1}\}$ generated by two elements $a$ and $b$. Then 
$$R:=\Delta\cup\{(a, b^{-1}), (b^{-1}, a)\}$$
 where $\Delta$ is the diagonal set is a semimosaic equivalence. However, $[a^{-1}]=\{a^{-1}\}$ whereas $[a]^{-1}=\{a, b^{-1}\}^{-1}=\{a^{-1}, b\}$.
\end{example}

\begin{proposition}\label{prop:congruences}
Let $M$ be a (semi)mosaic and $r:R\hookrightarrow M\times M$ be a congruence on $M$. Then, the quotient of $M$ by $R$ in $\sMsc$ has underlying set equal to the set-theoretic quotient of $M$ by $R$ iff $R$ is a (semi)mosaic equivalence on $M$.
\end{proposition}

\begin{proof}
    ($\impliedby$) Let $M$ be a semimosaic and $(r_1, r_2) = r:R\hookrightarrow M\times M$ be a congruence on $M$ that is also a semimosaic equivalence. Let $M/R$ be the quotient of $M$ by $R$ in $\Set$ and $p \colon M \twoheadrightarrow M/R$ be the quotient map that sends $x$ to its equivalence class $[x]$ under $R$. We wish to show that $p$ is the coequalizer of $r_1$ and $r_2$ in $\SMsc$ under some hyperoperation on $M/R$. By ~\cite[Proposition~3.3]{NakamuraReyes}, $p$ is the coequalizer of $r_1$ and $r_2$ in $\HMag$ under the following hyperoperation on $M/R$:
    \begin{align}
    \begin{split}
    [x] \star [y] &:= p(p^{-1}(x) \star p^{-1}(y)) \\
    &= \{[z] \mid z \in x' \star y' \mbox{for some } x' \in [x], y' \in [y]\}.
    \end{split}\label{eq:quotient product}
    \end{align}

We claim that $p(e) = [e]$ acts as the identity in $M/R$ under this hyperoperation. It is clear from~\eqref{eq:quotient product} that $[x]\in [x]\star[e] \cap [e]\star [x]$. On the other hand, if $[y] \in [x] \star[e] \cup [e]\star[x]$, then b~\eqref{eq:quotient product} we have $y' \in [y]$ and $x' \in [x]$ such that $y' \in x' \star [e] \cup [e] \star x'$. The semimosaic equivalence property of $R$ implies that $[y] = [y'] = [x'] = [x]$. So, we indeed have $[x] \star [e] = [x] = [e]\star[x]$. Thus~\eqref{eq:quotient product} turns $M/R$ into an object and $p$ into a morphism in $\SMsc$. Now, suppose that $p':M\to N$ is a morphism in $\SMsc$ such that $p'\circ r_1 = p'\circ r_2$. By the universal property of the coequalizer, there exists a unique morphism of \textit{hypermagmas} $q: M/R\to N$ such that $q\circ p = p'$. Thus, $q(p(e)) = p'(e)$ which, since $p$ and $p'$ are both morphisms of semimosaics, implies that $q([e]) = e$. Thus, $q$ is also a morphism of \textit{semimosaics}, giving us that $p$ is the coequalizer in $\SMsc$ of $r_1$ and $r_2$.

Now, let $M$ be a mosaic and assume that $R$ is a mosaic equivalence. Then, we claim that~\eqref{eq:quotient product} turns $M/R$ too into a mosaic with the inverse of $[x]$ being $[x^{-1}]$. Clearly, $[e]\in [x]\star[x^{-1}] \cap [x^{-1}]\star [x]$. We need to show reversibility. Let $x, y, z \in M$ be such that $[x]\in [y]\star[z]$ . By~\eqref{eq:quotient product}, there exist $x' \equiv_R x$, $y'\sim_R y$ and $z'\equiv_R z$ such that $x'\in y'\star z'$. By reversibility in $M$, $y'\in x' \star z'^{-1}$ which, since $R$ is a mosaic equivalence, implies that $[y]\in [x] \star [z^{-1}]$ (the proof for $[z]\in [y^{-1}]\star [x]$ is similar). Thus, under~\eqref{eq:quotient product}, $M/R$ is an object and $p$ a morphism in $\Msc$. Moreover, the same proof as above for why $p$ is the coequalizer of $r_1$ and $r_2$ in $\Msc$ applies because $\Msc$ is a full subcategory of $\SMsc$.

($\implies$) Now, let $M$ be a semimosaic and $(r_1, r_2)=r:R\hookrightarrow M\times M$ be a congruence on $M$ such that the quotient $M/R$ of $M$ by $R$ in $\SMsc$ has underlying set equal to the set-theoretic quotient of $M$ by $R$. Let $p:M\twoheadrightarrow M/R$ be the quotient map in $\SMsc$ and given $x\in M$, let $[x]$ be the (set-theoretic) equivalence class under $R$ of $x$. Then, the hypothesis implies that $p(x)=[x]$. Let $x, y\in M$ be such that $y\in x\star[e]$. We wish to show that $[y]=[x]$. We have $y\in x\star e'$ for some $e' \in [e]$. Applying $p$, we get
$$[y]=p(y) \in p(x\star e') \subseteq p(x)\cdot p(e') = [x]\cdot [e] = [x],$$
where $\cdot$ is the hyperoperation on $M/R$. The proof for the $y\in [e]\star x$ case is similar. 

If $M$ is a mosaic and $R$ is a congruence in $\Msc$, then we need to additionally show that $R$ is closed under inversion as a subset of $M\times M$. But, that is automatic as $R$ is, by assumption, taken to be a submosaic of the mosaic $M\times M$.
\end{proof}

The following example describes two congruences on $\K$ in $\SMsc$ whose quotients have underlying set equal to the set-theoretic quotient by the \textit{same} semimosaic equivalence. One of the congruences is effective while the other is not. It turns out that the effective congruence corresponds to a regular subsemimosaic of $\K\times \K$ while the non-effective one does not. This will motivate the next result that further characterizes the effective congruences in $\Msc$ and $\SMsc$.

\begin{example}
    The identity map $(i_1, i_2)=i:\Z_2\times \Z_2 \hookrightarrow \K\times \K$ is a congruence on $\K$ by Proposition~\ref{prop:congruences}. The equivalence relation on $\K$ here is the relation where all of $\K$ forms a class and, thus, is trivially a semimosaic equivalence. Let $p:\K\to \{0\}$ be the quotient map in $\Set$. Then, under ~\eqref{eq:quotient product}, $\{0\}$ is the semimosaic on one element and $p$ the trivial morphism. Moreover, $p$ is indeed the coequalizer of $i_1$ and $i_2$ in $\SMsc$. Note, however, that the congruence is not effective. Suppose it was the kernel pair of some morphism $f:\K\to N$ in $\SMsc$. Then, $f\circ i_1 = f\circ i_2$. But, as set maps, $i_1=i_1'$ and $i_2=i_2'$ where $(i_1', i_2')=i':\K\times \K \to \K\times \K$ is the identity map. Thus, $f\circ i_1' = f\circ i_2'$. By the universal property of the kernel pair, there exists a unique $g:\K\times \K \to \Z_2\times \Z_2$ in $\SMsc$ such that $i_j\circ g = i_j'$ for $j=1, 2$. Thus, $g$ must be the identity map from $\K\times \K$ to $\Z_2\times \Z_2$, but that is not a morphism in $\SMsc$. 
    
     Note also that $\Z_2 \times \Z_2$ is not a regular subsemimosaic of $\K \times \K$. This is because $(1, 1)+(1,1)$ in $\Z_2\times \Z_2$ is  $(0, 0)$ but $(1, 1)+ (1,1)$ in $\K\times \K$ is $\{0, 1\}\times \{0, 1\}$, which intersected with $\Z_2\times \Z_2$ as a set remains $\{0, 1\} \times \{0, 1\}$. $i:\K\times \K \to \K\times \K$, on the other hand, \textit{is} an effective congruence on $\K$ (it is the kernel pair of $p$ above) with the same corresponding semimosaic equivalence as above. $\K\times \K$ is also the unique regular subsemimosaic on the subset $\{0, 1\} \times \{0, 1\} \subseteq \K\times \K$. This motivates the following result.        
     \end{example}

    \begin{theorem}\label{thm:effective congruence}
        Let $M$ be a semimosaic (resp., mosaic). The effective congruences on $M$ are exactly the regular sub(semi)mosaics  $R \subseteq M \times M$ that are (semi)mosaic equivalences on $M$.
    \end{theorem}
    \begin{proof}
        We give the proof for semimosaics and that for mosaics will be the same. Let $M$ be a semimosaic and $R\subseteq M\times M$ be a semimosaic equivalence on $M$. Give $R$ the regular subsemimosaic structure obtained from $M\times M$. Let $(r_1, r_2)=r:R\hookrightarrow M\times M$ be the inclusion map, with $r_1$ and $r_2$ being the projections onto the first and second components of members of $R$ respectively. Let $p:M\twoheadrightarrow M/R$ be the quotient of $M$ by $R$ in $\SMsc$. We claim that $r_1, r_2$ is the kernel pair of $p$ in $\SMsc$. This is equivalent to saying that $r$ is the equalizer of $p\circ q_1$ and $p\circ q_2$ where $q_1, q_2: M\times M \to M$ are the projections onto the first and second components of members of $M$ respectively. Now, $r$ is the equalizer in $\Set$ of $p\circ q_1$ and $p\circ q_2$ because $M/R$ has underlying set equal to the set-theoretic quotient of $M$ by $R$ by Proposition~\ref{prop:congruences}. Suppose $s:S\to M\times M$ is a morphism in $\SMsc$ such that $(p\circ q_1) \circ s = (p\circ q_2)\circ s$. Then, by the universal property of the equalizer, there exists a unique \textit{set} map $t:S\to R$ such that $s=r\circ t$. We wish to show that $t$ is also a morphism in $\SMsc$. Let $a, b\in S$. Since $r$ is the inclusion, $s=r\circ t$ implies that the sets $t(a\cdot b) = s(a\cdot b)$ and $t(a)\star_R t(b) = s(a)\star_R s(b)$. Thus, it suffices to show that $s(a\cdot b) \subseteq s(a)\star_R s(b)$. Now, $s(a\cdot b) \subseteq s(a)\star_M s(b)$ since $s$ is a morphism in $\SMsc$, and $s(a\cdot b)\subseteq R$ as $s(S)\subseteq R$. Thus, $s(a\cdot b) \subseteq s(a)\star_M s(b) \cap R = s(a) \star_R s(b)$ by the regular subsemimosaic structure of $R$. Therefore, $t$ is indeed a morphism in $\SMsc$, which implies that $r$ is the equalizer of $p\circ q_1$ and $p\circ q_2$ in $\SMsc$. We have shown that $r$ is an effective congruence in $\SMsc$.
        
        Now, let $(r_1, r_2)=r:R\hookrightarrow M$ be a congruence on $M$ such that it is the kernel pair of some $f:M\to N$. Then, $r$ is an equalizer and, by~\cite[Theorem~3.14]{NakamuraReyes}, coshort. Thus, by \cite{NakamuraReyes} again, $R$ is isomorphic to the regular subsemimosaic $r(R)\subseteq M$. Moreover, by the discussion in \ref{subsection:regularcat}, $f$ is the coequalizer of $r_1$ and $r_2$ in both $\SMsc$ and $\Set$. Thus, by Proposition~\ref{prop:congruences}, $R$ is a semimosaic equivalence on $M$.        
    \end{proof}

\begin{corollary}\label{cor:mosaic quotients}
Let $M$ be a semimosaic (resp., mosaic). 
\begin{enumerate}
\item There is a bijective correspondence between
	\begin{itemize}
    \item effective congruences on $M$;
    \item semimosaic (resp., mosaic) equivalence relations on $M$;
	\item quotients objects of $M$ (i.e., isomorphism classes of regular epimorphisms with domain $M$);
	\end{itemize}
\item Every morphism $f \colon M \to N$ of (semi)mosaics factors uniquely as $f = i \circ p$, where $p \colon M \to M/R$ is the quotient of $M$ by an effective congruence, and $i \colon M/R \to N$ is an injective morphism.
\end{enumerate}
\end{corollary}

\begin{proof}
(1) The equivalence between effective congruences $R$ on an object $M$ and quotient objects of $M$ holds in every regular category as discussed in~\eqref{eq:quotient correspondence} above.
The equivalence between effective congruences and (semi)mosaic equivalence relations follows from Theorem~\ref{thm:effective congruence} and the fact that every subset of a semimosaic (resp., mosaic) that contains the identity (resp., and is closed under inverses) has a unique regular subsemimosaic (resp., submosaic) structure on it.

(2) Taking $R = R_f$ to be the kernel pair of $f$, this follows from~\eqref{eq:epimono} as discussed above.
\end{proof}

\begin{example}\label{ex:not exact}
The characterization above makes it straightforward to construct further examples of congruences on (semi)mosaics that are not effective, showing that the regular categories $\Msc$ and $\SMsc$ are not exact. 
For example, take the abelian group $M=\mathbb{Z}$ and define a hyperaddition on $R:=M\times M$ by
$$(a, b)\boxplus(c, d)=\left\{\begin{array}{llll}
(a, b) & \text{if}\ c=d=0\\
(c, d) & \text{if}\ a=b=0\\
(0, 0) & \text{if}\ c=-a, d=-b\\
\emptyset & \text{otherwise}
\end{array}\right.$$ 
Then $R$ is a congruence on $M$ that is not effective. Indeed, if it were effective, then the hypersum is always the componentwise sum.
\end{example}

\begin{remark}\label{rem:completeness}
Let $M$ be a (semi)mosaic. Let $E \subseteq \P(M \times M)$ denote the comlpete lattice of equivalence relations on $M$, and let $E \subseteq L$ denote the subset of (semi)mosaic equivalences on $M$.
It is evident that the set of (semi)mosaic equivalences on $M$ is closed under intersections in $\P(M \times M)$. The infimum operation of $E$ is given by intersection, so this shows that $L$ is a complete sub-meet-semilattice of $E$.

It follows by abstract order theory that $L$ is a complete lattice. However, it is not generally a sublattice of $E$. Indeed, the join of $E$ is given by the transitive closure of the union of equivalence relations.
Consider the commutative mosaic $M$ whose hyperoperation is given by the following table:
\begin{center}
\begin{tabular}{ |c|c|c|c|c|c| } 
\hline
$\star$ & $e$ & $a$ & $b$ & $c$ & $-c$ \\
\hline
$e$ & $e$ & $a$ & $b$ & $c$ & $-c$ \\ 
\hline
$a$ & $a$ & $0$ & $\emptyset$
& $\emptyset$ & $\emptyset$\\ 
\hline
$b$ & $b$ & $\emptyset$ & $0$
& $-c$ & $c$\\ 
\hline
$c$ & $c$ & $\emptyset$ & $-c$
& $\emptyset$ & $0$\\ 
\hline
$-c$ & $-c$ & $\emptyset$ & $c$
& $0$ & $\emptyset$\\ 
\hline
\end{tabular}
\end{center} 
Define equivalence relations $\equiv_1, \equiv_2$ on $M$ by the respective subsets
$$\Delta\cup\{(a, 0), (0, a)\},\ \Delta\cup\{(a, b), (b, a)\}$$
where $\Delta\subset M\times M$ is the diagonal set. These are mosaic congruences. Now we have 
$$b\equiv_2 a,\ a\equiv_1 0,\ -c\in b+c.$$
However, $c$ and $-c$ are not equivalent under the transitive closure of the union of $\equiv_1$ and $\equiv_2$.
\end{remark}

\subsection{Further (non-)exactness properties for mosaics}
\label{sub:exact}

The existence of non-effective congruences as in Example~\ref{ex:not exact} shows that the categories $\Msc$ and $\SMsc$ are not Barr exact. This is in contrast with the categories of groups and monoids, both of which are Barr exact.
In this section we discuss a few other ``exactness'' properties of categories and their relationships to mosaics and semimosiacs.
We begin with some properties that fail to hold for (semi)mosaics.

\begin{remark}\label{rem:malcev}
Another property of categories that lies strictly between regular and exact categories is that of finitely complete \emph{protomodular} categories~\cite{Bourn}. Without recalling the definition here, we recall from~\cite[Theorem~3.18]{BournGran} that finitely complete protomodular categories are \emph{Malcev}, meaning that every reflexive relation is a congruence. 
We can show that the (finitely complete) categories $\SMsc$ and $\Msc$ are not Malcev by considering the regular subobject
\[
R = \{(0, 0), (1, 1), (1, 0)\} \hookrightarrow \Z_2 \times \Z_2
\]
with the empty sum $(1,0) + (1,1) = \varnothing$. This is reflexive but not symmetric and therefore is not a congruence on $\Z_2$. 
\end{remark}

The notion of a proto-exact structure on a category was introduced by Dyckerhoff and Kapranov~\cite[\S 2.4]{DK} in order to provide a non-additive generalization of Quillen exact categories. Since then, proto-exact structures have been found in categories of matroids~\cite{EJS}, modules over semirings~\cite{JST}, and further generalizations thereof~\cite{JSW}.
We briefly recall the definition here.

A \emph{proto-exact category} $(\C, \mathfrak{M}, \mathfrak{E})$ is a pointed category $\C$ along with two classes $\mathfrak{M}$ and $\mathfrak{E}$ of morphisms in $\C$, respecitvely called the \emph{admissible monomorphisms} and \emph{admissible epimorphisms}, satisfying the following axioms:
\begin{itemize}
    \item $\mathfrak{M}$ and $\mathfrak{E}$ are closed under composition and contain all isomorphisms;
    \item All morphisms $0 \to X$ are in $\mathfrak{M}$ and all morphisms $X \to 0$ are in $\mathfrak{E}$;
    \item any commutative square
    \[
        \begin{tikzcd}
            A \ar[r, "i_1"] \ar[d, "p_1"] & B \ar[d, "p_2"] \\
            C \ar[r, "i_2"] & D
        \end{tikzcd}
    \]
    with $i_1, i_2 \in \mathfrak{M}$ and $p_1, p_2 \in \mathfrak{E}$ is a pullback if and only if it is a pushout;
    \item $\mathfrak{E}$ has pullbacks along $\mathfrak{M}$ and is closed under such pullbacks;
    \item $\mathfrak{M}$ has pushouts along $\mathfrak{E}$ and is closed under such pushouts.
\end{itemize}
These axioms imply that $\mathfrak{M}$ consists of normal monomorphisms and $\mathfrak{E}$ consists of normal epimorphisms. For instance, if $i_1 \colon A \to B$ is in $\mathfrak{M}$, then taking the pushout along $p_1 \colon A \to 0$ yields a bicartesian square as above where $C = 0$, from which it follows that $i_1 = \ker p_2$.
If we assume from the outset that $\mathfrak{M}$ and $\mathfrak{E}$ respectively contain only normal monomorphisms and normal epimorphisms, it is shown in~\cite[Lemmas~1.2]{Mozgovoy} that the third axiom above is redundant. In case $\C$ has kernels and cokernels, it also follows from~\cite[Lemma~1.3]{Mozgovoy} that the existence of the pullbacks and pushouts of the fourth and fifth axioms are automatic. 

Recently, Mozgovoy introduced the following notion of a category for which the classes of normal monomorphisms and normal epimorphisms yeild a proto-exact structure. We recall the defintion below. 

\begin{definition}
A pointed category $\catC$ is defined to be \emph{parabelian}~\cite{Mozgovoy} if it satisfies the following conditions:
\begin{enumerate}[label=(\roman*)]
    \item $\catC$ has kernels and cokernels; 
    \item  normal epimorphism  are closed under pullbacks along normal monomorphisms;
    \item normal monomorphisms are closed under pushouts along a normal epimorphisms.
\end{enumerate}
(By~\cite[Lemma~1.3]{Mozgovoy}, condition~(i) implies that the pullbacks and pushotus of~(ii) and~(iii) necessarily exist.)
\end{definition}

It is shown in~\cite[Theorem~1.10]{Mozgovoy} that every parabelian category is proto-exact if we take $(\mathfrak{M}, \mathfrak{E})$ to be the classes of normal monomorphisms and normal epimorphisms. In the remainder of this section, we verify that the category of mosaics is parabelian.

For the next several results, we will let $\cC$ be one of $\SMsc, \Msc$ or $\cMsc$ and let $\mathfrak{M}$ (resp., $\mathfrak{E}$) denote the class of normal monomorphisms (resp., normal epimorphisms) in $\cC$. If $\cC\in\{\SMsc, \Msc, \cMsc\}$, then the normal monomorphisms (resp. the normal epimorphisms) in $\cC$ are exactly the strict monomorphisms 
(resp. the 
cokernels of strict absorptive subsemimosaics or strict submosaics, as in Lemma~\ref{lem:cokernel})
by \cite[Theorem 1.2]{NakamuraReyes}. This also holds for $\cC=\cSMsc$ by the proof of \cite[Theorem 3.16]{NakamuraReyes} and the fact that commutative submosaics are closed under strict subsemimosaics and cokernels.

\begin{lemma}\label{lem:pullback}
Let $\cC\in\{\SMsc, \cSMsc, \Msc, \cMsc\}$. Consider the following pullback diagram in $\cC$ induced by $f\in\mathfrak{M}$ and $g\in\mathfrak{E}$:
\begin{center}
\begin{tikzcd}
	{A\times_{C}B} && B \\
	\\
	{A} && {C}
	\arrow["g", two heads, from=1-3, to=3-3]
	\arrow["f"', hook, from=3-1, to=3-3]
	\arrow["{p_1}"', from=1-1, to=3-1]
	\arrow["{p_2}", from=1-1, to=1-3]
\end{tikzcd}
\end{center}
Then $p_1$ is a cokernel of the strict subobject 
$$D=\{(1_A, x')|g(x')=1_C\}\subset A\times_CB.$$
Therefore, normal epimorphisms are closed under pullbacks by normal monomorphisms.
\end{lemma}
\begin{proof}

Let $\alpha:A\times_CB\rightarrow A'$ be a morphism such that $\alpha(D)=1_{A'}$. We define $\beta:A\rightarrow A'$ by $\beta(x):=\alpha(x, x')$ where $x'\in B$ is any element such that $(x, x')\in A\times_CB$.

We must show that $\beta$ is well-defined and a morphism, from which it will follow that it is the unique morphism making the diagram
\[\begin{tikzcd}
	D && {A\times_CB} && A \\
	\\
	&&&& {A'}
	\arrow[hook, from=1-1, to=1-3]
	\arrow["{p_1}", two heads, from=1-3, to=1-5]
	\arrow["\alpha", from=1-3, to=3-5]
	\arrow["{\exists!\beta}", dashed, from=1-5, to=3-5]
\end{tikzcd}\]
commute. If $(x, x'), (x, y')\in A\times_CB$, then we have $g(x')=f(x)=g(y')$. Since $g$ is a cokernel, by Lemma~\ref{lem:cokernel} there exists a sequence $x'=z_0, z_1, \cdots, z_n=y'$ in $B$ such that for every $i$ there exists some $e_i\in\ker g$ satisfying 
$$z_i\in z_{i+1}\star e_i\cup e_i\star z_{i+1}\ \text{or}\  z_{i+1}\in z_{i}\star e_i\cup e_i\star z_{i}$$
(note that a cokernel is the cokernel of its kernel). By an inductive argument, $g(z_i)=g(z_0)=f(x)$, i.e., $(x, z_i)\in A\times_CB$ for all $i$.
If $z_i\in z_{i+1}\star e_i$ for example, then we have
$$\alpha(x, z_i)\in\alpha(x, z_{i+1}\star e_i)\subset \alpha(x, z_{i+1})\star\alpha(1_A, e_i)=\alpha(x, z_{i+1})\star 1_{A'}=\alpha(x, z_{i+1})$$
where the first equality follows from the assumption that $\alpha(D)=1_{A'}.$
The other cases follow similarly. An inductive argument shows that $\alpha(x, x')=\alpha(x, y')$. It is straightforward to show that $\beta$ is indeed a morphism in $\cC$ and that it is unique morphism that makes the diagram commute. 
This verifies that $p_1$ is the cokernel of the inclusion $D \hookrightarrow A \times_C B$ as desired.
\end{proof}

Our next goal is to show that the normal monomorphisms are closed under pushouts along normal epimorphisms (Corollary~\ref{cor:pushout}). This requires the following lemma on $\HMag$.

\begin{lemma}\label{HML}
Consider the following pushout diagram in $\HMag$ induced by a homomorphism $f \colon C\rightarrow A$ and an injective strict homomorphism $g \colon C\hookrightarrow B$:
\begin{center}
\begin{tikzcd}
	C && B \\
	\\
	A && {A+_CB}
	\arrow["{i_2}", from=1-3, to=3-3]
	\arrow["{i_1}"', from=3-1, to=3-3]
	\arrow["f"', from=1-1, to=3-1]
	\arrow["g", hook, from=1-1, to=1-3]
\end{tikzcd}
\end{center}
Then the underlying set of the pushout $P:=A+_CB$ can be described as $A\coprod (B\setminus g(C))$. The induced homomorphism
$i_1:A\hookrightarrow P$ is the canonical injective map and the homomorphism $i_2:B\rightarrow P$ is given by
$$i_2(b)=\left\{\begin{array}{llll}
f(c) & \text{if}\ b=g(c)\ \text{for unique}\ c\in C\\
b & \text{otherwise}
\end{array}\right.$$
The hyperoperation on $P$ as a pushout $A+_CB$ is given by
$$x\star_P y=\left\{\begin{array}{llll}
x\star_Ay & \text{if }\ x, y\in A\\
i_2(g(f^{-1}(x))\star_By) & \text{if }\ x\in A,\ y\in B\setminus g(C)\\
i_2(x\star_Bg(f^{-1}(y))) & \text{if }\ x\in B\setminus g(C),\ y\in A\\
i_2(x\star_By) & \text{if }\ x, y\in B\setminus g(C)
\end{array}\right.$$
\end{lemma}
\begin{proof}
Let $P$ be the hypermagma defined above.
The canonical injective map $i_1 \colon A\hookrightarrow P$ is certainly a morphism of hypermagmas. To see that the map $i_2 \colon B\rightarrow P$ defined above is a morphism of hypermagmas, let $b_1$ and $b_2$ be elements of $B$. If $b_1=g(c_1)$ and $b_2=g(c_2)$ for some $c_1, c_2\in C$, then using the fact that $g$ is strict we have
\begin{eqnarray*}
i_2(b_1\star_B b_2)&=&i_2(g(c_1)\star_B g(c_2))=i_2\circ g(c_1\star_C c_2)=f(c_1\star_C c_2)\\
&\subseteq& f(c_1)\star_Af(c_2)=i_2(b_1)\star_Ai_2(b_2)=i_2(b_1)\star_P i_2(b_2).
\end{eqnarray*}
If $b_1\notin g(C)$ or $b_2\notin g(C)$, then the definition of the hyperoperation on $P$ tells us that the product $i_2(b_1)\star_Pi_2(b_2)$ contains $i_2(b_1\star_B b_2)$. Thus $i_2:B\rightarrow P$ is a homomorphism of hypermagmas.

Now to prove the claim, it suffices to check the universal property for the hypermagma $P$ and morphisms $i_1$, $i_2$ defined above.
By definition $i_1\circ f=i_2\circ g$. Suppose that we have morphisms $f':A\rightarrow D$ and $g':B\rightarrow D$ such that $f'\circ f=g'\circ g$. We define $h:P=A\coprod (B\setminus g(C))\rightarrow D$ by
$$h(x):=\left\{\begin{array}{llll}
f'(x) & \text{if}\ x\in A\\
g'(x) & \text{if}\ x\in B\setminus{g(C)}
\end{array}\right.$$
\[\begin{tikzcd}
	C && B \\
	\\
	A && {P} \\
	\\
	&&&& D
	\arrow["g", hook, from=1-1, to=1-3]
	\arrow["f"', from=1-1, to=3-1]
	\arrow["{i_1}"', from=3-1, to=3-3]
	\arrow["{i_2}", from=1-3, to=3-3]
	\arrow["{f'}"', from=3-1, to=5-5]
	\arrow["{g'}", from=1-3, to=5-5]
	\arrow["{\exists!h}", dashed, from=3-3, to=5-5, pos=0.3]
\end{tikzcd}\]
By definition $h$ satisfies $h\circ i_1=f'$ and $g'=h\circ i_2$ and it is the unique map satisfying these relations. We must check that $h$ is a morphism of $\HMag$. If $x, y\in A$ or $x, y\in B\setminus g(C)$, we easily see that $h(x\star y)\subseteq h(x)\star_D h(y)$. Suppose that $x\in A$ and $y\in B\setminus g(C)$. Then
\begin{align*} 
h(x\star_P y)&=h\circ i_2(g(f^{-1}(x)\star_By))=g'(g(f^{-1}(x)\star_By)) \\
& \subseteq g'\circ g(f^{-1}(x))\star_Dg'(y)
=f'\circ f(f^{-1}(x))\star_Dg'(y)\subset f'(x)\star_D g'(y) \\
&=h(x)\star_Dh(y)
\end{align*}
The other case is similar. 
\end{proof}

The next few results use the unitization construction of~\cite[Definition~3.7]{NakamuraReyes}. If $M$ is a hypermagma and $E \subseteq H$ is a subset, the \emph{unitization of $M$ with respect to $E$} is a semimosaic $M_E$ (with identity $e$) equipped a morphism $\pi_E \colon M \to M_E$ that is universal with respect to the property $E \subseteq \pi^{-1}(e)$. The set $\pi^{-1}(e)$ is in fact the smallest strict absorptive subhypermagma of $M$ containing $E$. For details of its construciton, see~\cite[Lemma~3.10]{NakamuraReyes}.

\begin{corollary}
    In the categories $\SMsc$ and $\cSMsc$, normal monomorphisms are not closed under pushout along normal epimorphisms. Therefore, they are not parabelian.
\end{corollary}
\begin{proof}
    Let $F_2:=\{e, a, c\}$ and $F_1:=\{e, c\}$ be free semimosaics where $e$ is the identity. Let $f \colon F_2\rightarrow F_1$ be the morphism defined by sending $a\mapsto e$ and $c\mapsto c$. Let $E:=\{e, a, b, c\}$ be a semimosaic where $e$ is the identity,  $a\star b=c$, $c\star b=e$, and all other products are empty. Let $g$ denote the inclusion $F_2\hookrightarrow E$. The morphism $f$ (resp. $g$) is a strict epimorphism (resp. monomorphism) in $\SMsc$.
\[\begin{tikzcd}
	F_2={\{e, a, c\}} & E={\{e, a, b, c\}} \\
	F_1={\{e, c\}} & P={\{e, b, c\}}
	\arrow["g", hook, from=1-1, to=1-2]
	\arrow["f"', two heads, from=1-1, to=2-1]
	\arrow["{i_2}", from=1-2, to=2-2]
	\arrow["{i_1}"', from=2-1, to=2-2]
\end{tikzcd}\]
The pushout of $f$ and $g$ in $\HMag$ is given by $P:=\{e, b, c\}$ where, $i_1$ is the inclusion and $i_2$ sends $a\mapsto e$ and preserves other elements. Note that the pushout $P_u$ of $f$ and $g$ in $\SMsc$ is the unitization of $P$ with respect to the subset $\{e\}\subset P$. This is the same as the unitization of $P$ by the smallest strict absorptive subhypermagma of $P$ containing $e$.

By Lemma~\ref{HML}, 
$$e\star_P b=i_2(g(f^{-1}(e))\star b)=i_2(g(\{e, a\})\star b)=i_2(b, c)=\{e, c\}$$
and
$$c\star_Pb=i_2(g(f^{-1}(c))\star b)=i_2(g(\{c\})\star b)=i_2(e)=e.$$
Therefore, the smallest absorptive strict subhypermagma of $P$ containing $e$ is $P=\{e, b, c\}$. Thus the pushout $P_u$ in $\SMsc$ is trivial and the canonical homomorphism $F_1\rightarrow P_u$ is not a normal monomorphism in $\SMsc$. A similar argument works for $\cSMsc$ by taking $E={e, a, b, c}$ where $e$ is the identity, $a\star b=b\star a=c, c\star b=b\star c=e$ and other products are empty.
\end{proof}

\begin{corollary}\label{cor:pushout}
Consider the following pushout diagram in $\cC \in \{\Msc, \cMsc\}$ induced by $f\in\mathfrak{E}$ and $g\in\mathfrak{M}$:
\begin{center}
\begin{tikzcd}
	C && B \\
	\\
	A && {A+_CB}
	\arrow["{i_2}", from=1-3, to=3-3]
	\arrow["{i_1}"', from=3-1, to=3-3]
	\arrow["f"', two heads, from=1-1, to=3-1]
	\arrow["g", hook, from=1-1, to=1-3]
\end{tikzcd}
\end{center}
Then $i_1\in\mathfrak{M}$. Therefore, normal monomorphisms are closed under pushout by normal epimorphisms.
\end{corollary}
\begin{proof}
By following the proofs of \cite[Theorems~3.11, 4.1]{NakamuraReyes}, we see that the pushout $P_u:=A+_CB$ in $\cC$ is the unitization of the pushout $P=A\coprod (B\setminus g(C))$ in $\HMag$ with respect to the element $1_A\in P$. Note that the resulting semimosaic is naturally endowed with a mosaic structure, as in the proof of~\cite[Theorem~4.1]{NakamuraReyes}.

We claim that the subhypermagma $\{1_A\} \subseteq P$ is absorptive. This means that for $x \in P$, if $1_A \in x \star_P 1_A \cup 1_A \star_P x$, then $x = 1_A$. If $x \in A$, this is easily verified since $A$ is a strict subhypermagma of $P$. On the other hand, if $x \in B \setminus g(C)$ we claim that $x\star_P1_A\cup 1_A\star_Px\subset B\setminus g(C)$; this will complete the claim since $1_A \notin B \setminus g(C)$.
Note that by the description of the multiplication on $P$ given by Lemma~\ref{HML}, we have 
$$x\star_P1_A=i_2(x\star_Bg(f^{-1}(1_A)))\subset P.$$
To see that this is a subset of $B\setminus g(C)$, it suffices to show that the subset $x\star_Bg(f^{-1}(1_A))\subset B$ is indeed a subset of $B\setminus g(C)$. Suppose toward a contradiction that the subset $x\star_Bg(f^{-1}(1_A))$ contains $g(z)$ for some $z\in C$. This means that there exists $y\in C$ such that $g(z)\in x\star_Bg(y)$ and $f(y)=1_A$. The reversibility axiom and the strictness of $g$ imply that $x$ lies in $$g(z)\star_Bg(y)^{-1}=g(z)\star_Bg(y^{-1})=g(z\star_C y^{-1}),$$
which is a contradiction since $x\notin g(C)$.
Therefore the subset $x\star_Bg(f^{-1}(1_A))$ is contained in $B\setminus g(C)$ and hence $x\star_P1_A\subset B\setminus g(C)$. A symmetric argument shows that $1_A\star_Px\subset B\setminus g(C)$.

Now that we know $\{1_A\}$ is a strict absorptive subhypermagma of $P$, it follows from
the construction in~\cite[Lemma 3.10]{NakamuraReyes}, two elements $x, y\in P$ are identified in $P_u$ if there exists a sequence $x=z_1, \cdots, z_n=y$ such that 
$$z_{i+1}\in z_i\star_P 1_A\cup 1_A\star_P z_i\ \text{or}\ z_{i}\in z_{i+1}\star_P 1_A\cup 1_A\star_P z_{i+1}.$$
If $x, y$ belong to the strict subhypermagma $A\subset P$, then we inductively see that 
$x=z_1=\cdots=z_n=y$. This means 
the equivalence relation on $P=A\coprod B\setminus g(C)$ whose quotient defines the unitization $P_u$ restricts to the trivial relation on the subset $A$. Therefore the canonical homomorphism $A\rightarrow A+_CB$ is injective. The strictness follows from the definition of the multiplication on $P$ given by Lemma~\ref{HML} and by the way the multiplication of the unitization is defined in the proof of~\cite[Lemma~3.10]{NakamuraReyes}.
\end{proof}

This immediately yields our desired result.

\begin{theorem}\label{thm:parabelian}
The categories $\Msc$ and $\cMsc$ are parabelian. Thus they are proto-exact with respect to the classes of normal monomorphisms and normal epimorphisms.
\end{theorem}

\begin{proof}
    This follows immediately from Lemma~\ref{lem:pullback} and Corollary~\ref{cor:pushout}.
\end{proof}

Recall the following examples of proto-exact categories from the literature:
\begin{itemize}
\item   $(\Set_*, \mathfrak{M}_{\Set_*}, \mathfrak{C}_{\Set_*})$, where $\mathfrak{M}_{\Set_*}$ is the class of injective morphisms and $\mathfrak{C}_{\Set_*}$ is the class of surjective morphisms that are injective outside the preimages of the basepoints of the codomains (\cite[Example 2.4]{JST}).
\item $(\Can, \mathfrak{M}_{\Can}, \mathfrak{C}_{\Can})$, where $\mathfrak{M}_{\Can}$ is the class of strict monomorphisms and $\mathfrak{C}_{\Can}$ strict epimorphisms (\cite[Theorem 5.11]{JST}).
\end{itemize}  
These categories have fully faithful embeddings into $\cMsc$ described as follows. There is a free functor $F \colon \Set_*\hookrightarrow \cMsc$, which sends every pointed set $X=(X, *)$ to the “free" mosaic whose underlying pointed set is 
$$FX=\{*\}\cup (X\backslash \{*\})\cup -(X\backslash \{*\})$$ 
(a similar construction can be found in the proof of \cite[Theorem 4.3]{NakamuraReyes}).
We also have the inclusion functor $I \colon \Can \hookrightarrow \cMsc$ of the strict subcategory of canonical hypergroups into commutative mosaics. 
The following shows that these embeddings are compatible with the proto-exact structure of each category. Let $\mathfrak{M}_{\cMsc}$ (resp. $\mathfrak{C}_{\cMsc}$) denote the class of normal monomorphisms (resp. normal epimorphisms) in $\cMsc$.

\begin{theorem}
Retain the notation introduced above.
\begin{enumerate}
\item Let $f$ be a morphism in $\Set_*$. Then $f\in\mathfrak{M}_{\Set_*}$ if and only if $F(f)\in\mathfrak{M}_{\cMsc}$. Also, $f\in\mathfrak{C}_{\Set_*}$ if and only if $F(f)\in\mathfrak{C}_{\cMsc}$. 
\item Let $f$ be a morphism in $\Can$. Then $f\in\mathfrak{M}_{\Can}$ if and only if $I(f)\in\mathfrak{M}_{\cMsc}$. Also, $f\in\mathfrak{C}_{\Can}$ if and only if $I(f)\in\mathfrak{C}_{\cMsc}$. 
\end{enumerate}
\end{theorem}
\begin{proof}
(1) The first assertion follows from the fact that every submosaic of $FY$ is strict. The second assertion follows from the fact that every subset of $FX$ containing the identity and closed under negation of is a strict submosaic.

(2) The first assertion is immediate. We show the second assertion. If $f:X\rightarrow Y$ is a strict surjective homomorphism then $Y$ is isomorphic to the quotient hypergroup $X/Z$ for some sub hypergroup $Z\subset X$. Thus $I(f)$ is a normal epimorphism in $\cMsc$. Conversely, if $I(f)$ is a normal epimorphism in $\cMsc$, then $f$ is a unitization of $X$ relative to some strict subhypergroup $Z\subset X$. The equivalence relation $\sim$ defining the unitization is given as 
$$x\sim x' \overset{\mathrm{def}}{\Longleftrightarrow} x\in x'+z_1+\cdots+z_n,\ \text{for some}\ z_1, \cdots, z_n\in Z$$
by \cite[Lemma 3.10]{NakamuraReyes}. We can easily see that if $x_1\sim x_1'$ and $x_2\sim x_2'$ then $x_1+x_2\sim x_1'+x_2'$, which means the unitization morphism $f$ is strict by (the argument on the additive structure of) \cite[Proposition 3.13]{Jun:geometry}. (It also follows from Proposition~\ref{prop:strict epi} which will be proved later.)
\end{proof}

\begin{remark}\label{rem:proto-abelian}
Another type of category with a canonical proto-exact structure is that of a \emph{proto-abelian} category; see~\cite{Andre, Dyckerhoff} and the comparison of these two definitions in~\cite[Remark~1.13]{Mozgovoy}. These are pointed categories for which the classes of \emph{all} monomorphisms and epimorphisms yields a proto-exact structure. For these categories, every monomorphism must be a kernel and every epimorphism must be a cokernel. This fails for the categories $\Msc$ and $\cMsc$ due to the existence of non-normal monomorphisms and epimorphisms. Thus these are examples of parabelian categories that are not proto-abelian. 
\end{remark}

\section{Special cases and examples}
\label{sec:examples}

In this final section, we describe several examples of effective congruences on mosaics and semimosaics and their quotients. These include both general constructions and some explicit computations.

\subsection{Quotients by endomorphisms and automorphisms}

A common way to produce hyperstructures is by taking quotients of ordinary algebraic structures by group actions. Here we show that a similar principle holds for mosaics and semimosaics. In fact, we can even take quotients by \emph{endomorphisms} rather than automorphisms, an observation that will be required in~\cite{NakamuraReyes:Hoops}.

If $M$ is a set and $\phi \colon M \to M$ is any function, we let ${\equiv_\phi}$ denote the equivalence relation on $M$ generated by $x \equiv_\phi \phi(x)$ for all $x \in M$. If $\phi$ is bijective, this is simply the $\phi$-orbit relation. For general $\phi$, this can be described by
\[
x \equiv_\phi y \iff \phi^m(x) = \phi^n(y) \mbox{ for some } m,n \geq 0.
\]

\begin{proposition}\label{prop:orbit equivalence}
Let $M$ be a (semi)mosaic, and fix $\phi \in \End(M)$. Then $\equiv_\phi$ is a (semi)mosaic equivalence.
\end{proposition}

\begin{proof}
Fix the identity $e \in M$
Since $\phi(e) = e$, it is straightforward to check that
\[
[e] = \{z \in M \mid \phi^n(z) = e \mbox{ for some } n \geq 0\}.
\]
To see that $\equiv_\phi$ is a (semi)mosaic congruence, suppose $x,y \in M$ with $y \in x \star [e]$. Then there exists $z \in M$ and $n \geq 0$ such that $\phi^n(z) = e$ and $y \in x \star z$. Because $\phi^n \in \End(M)$ is also a morphism of (semi)mosaics, we have
\begin{align*}
\phi^n(y) \in \phi^n(x \star z) &\subseteq \phi^n(x) \star \phi^n(z) \\
&= \{\phi^n(x)\},
\end{align*}
so that $\phi^n(x) = \phi^n(y)$. This proves that $x \equiv_\phi y$, and a symmetric argument shows that $y \in [e] \star x$ implies $y \equiv_\phi x$. Finally, if $M$ is a mosaic and $x \equiv_\phi y$, it is easy to see that $x^{-1} \equiv_\phi y^{-1}$. This completes the proof.
\end{proof}

\begin{example}
For $d \in \Z$, let $\phi_d \in \End(\Z) \cong \Z$ be the endomorphism given by multiplication $\phi_d(x) = dx$, and let $\equiv_d$ denote the congruence $\equiv_{\phi_d}$. Note that $\equiv_0$ is the universal relation and $\equiv_1$ is the diagonal relation. The quotient $\Z/\equiv_{-1}$ is the same as the additive hypergroup of the quotient hyperring $\Z/\{\pm 1\}$, which is given by $[x]+[y] = \{[x+y], [x-y]\}$.

So assume from now on that $|d| > 1$. Let $\langle d \rangle \leq \Q^*$ denote the rational multiplicative group generated by $d$. Note that $[0]_d = \{0\}$ is a singleton equivalence class. For $x,y \in \Z \setminus \{0\}$ we have
\begin{align*}
x \equiv_d y &\iff d^m x = d^n y \mbox{ for some } m,n \geq 0 \\
&\iff x/y \in \langle d \rangle.
\end{align*}
Each equivalence class $[x]_d$ has a unique representative that is not divisible by~$d$; it is equivalently the unique integer element of the $\langle d \rangle$-orbit of $x$ having minimal absolute value. We refer to this element as the \emph{$d$-free part} of $x$.
In case $x$ and $y$ are not divisible by $d$, we have $x \equiv_d y$ if and only if $x = y$. This demonstrates that $\Z/\equiv_d$ is an infinite mosaic for all $d \neq 0$.

Suppose that $d \nmid x,y \neq 0$. Then the elements of $[x]_d \ostar [y]_d$ are those $[z]_d$ where $z = d^ix + d^jy$ for some $i,j \geq 0$. The unique representatives of these classes $[z]_d$ can be taken to be of the form $x+d^jy$ or $d^ix + y$ for $i,j \geq 1$, while $x+y$ may be divisble by $d$ in which case $[x+y]=[z_{x,y}]$ for the $d$-free part $z_{x,y}$ of $x+y$. 
The distinct elements $x+d^j y$ are not divisible by $d$ and thus all represent distinct classes for $j \geq 1$, as do the elements $d^i x + y$ for $i \geq 1$. This shows that 
\[
[x]_d \ostar [y]_d = \{[x+d^jy]\}_{j \geq 1} \cup \{[d^i x+y]\}_{i \geq 1} \cup \{[x+y]\}
\]
is infinite. (Note that the first two sets in the union need not be disjoint, as in the extreme case where $x = y$.)
\end{example}

A quotient construction that is common in the theory of hyperrings~\cite{Krasner:quotient} is the quotient by a group of automorphisms. The same construction works for mosaics and semimosaics as follows.

\begin{corollary}
Let $M$ be a (semi)mosaic, and let $G$ be a group acting on $M$ by automorphisms. Then the $G$-orbit equivalence is a (semi)mosaic equivalence, and the quotient by this equivalence relation is the orbit space $M/G$ endowed with the following hyperoperation:
\[
(Gx) \ostar (Gy) = \{Gz \mid z \in gx \star hy \mbox{ for some } g,h \in G\}.
\]
\end{corollary}

\begin{proof}
The $G$-orbit equivalence $\equiv_G$ is the union of the $\phi$-orbit equivalences $\equiv_\phi$ for each $\phi \in G$. Therefore, if $y\in x \star u$ for some $u\equiv_G e$ (where $e$ is the identity of $M$), then $u \equiv_{\phi} e$ for some $\phi \in G$. Since $\equiv_\phi$ is a (semi)mosaic equivalence by Proposition~\ref{prop:orbit equivalence}, we have $y\equiv_\phi x$, which implies that $y\equiv_G x$.

The quotient by $\equiv_G$ is the usual map $p \colon M \twoheadrightarrow M/G$ sending elements $m \in M$ to orbits $Gm$. The description of the hyperoperation $\ostar$ on $M/G$ above now follows from the fact that $p$ is a short morphism.
\end{proof}

\subsection{Strict epimorphisms}

Recall that a morphism of semimosaics $p \colon M \to N$ is \emph{strict} if it satisfies 
\[
p(x \star y) = p(x) \star p(y)
\]
for all $x,y \in M$.
It is easy to verify~\cite[Lemma~2.13]{NakamuraReyes} that every strict surjective morphism is also short and thus regular, so is isomorphic to a quotient by an effective congruence. In this section we characterize those effective congruences whose quotient is a strict morphism.

In order to characterize strict epimorphisms as quotients by relations, we recall the following notion from~\cite[Definition~3.10]{Jun:geometry}. Given an equivalence relation $\equiv$ on $M$ and subsets $X,Y \subseteq M$, we write 
\[
X \equiv Y \quad \iff \quad 
\begin{array}{l} 
    \mbox{for all } x \in X \mbox{ and } y \in Y, \\
    \mbox{there exist } x' \in X \mbox{ and } y' \in Y \\
    \mbox{such that } x \equiv y' \mbox{ and } y \equiv x'.
\end{array}
\]

For a congruence $R \subseteq M \times M$, we will alternately denote $xRy$ as $x \equiv_R y$ for ease of reading. This notation will translate to congruence of subsets (in the sense above) as $X \equiv_R Y$.

\begin{definition}
Let $M$ be a mosaic. A strict submosaic $L \subseteq M$ is \emph{normal} if it satisfies the following conditions for all $x, y \in M$:
\begin{enumerate}[label=(\roman*)]
    \item $x \star L = L \star x$;
    \item $(x \star L) \star (y \star L) \subseteq (x \star y) \star L$.
\end{enumerate}
\end{definition}

In the definition above, note that if $M$ is commutative then~(i) always holds, and if $M$ is associative then (i)$\implies$(ii). In particular, if $M$ is a group then the condition above is the same as the usual notion of a normal subgroup.\footnote{Note that this term is not intended to suggest any connection with \emph{normal morphisms}.}

\begin{lemma}\label{lem:normal quotient}
Let $L$ be a normal submosaic of a mosaic $M$. The cosets $L \star x = x \star L$ partition $M$.
The set $M/L$ of cosets forms a mosaic under the hyperoperation
\[
(x \star L) \ostar (y \star L) := \{z \star L \mid z \in x \star y\},
\]
with the quotient map $q \colon M \twoheadrightarrow M/L$ being a strict morphism.
\end{lemma}

\begin{proof}
We will show that the relation defined on $M$ by 
\[
x \equiv y \iff x \in y \star L
\]
is an equivalence relation. It is reflexive because $e_M \in L$, and it is symmetric by reversibility. To see transitivity, suppose that that $x,y,z \in M$ are such that $x \in y \star L$ and $y \in z \star L$. Then
\[
x \in y \star L \subseteq (z \star L) \star (e \star L) \subseteq (z \star e) \star L = z \star L.
\]
This shows that ``belonging to the same coset'' is an equivalence relation, so that the cosets partition $M$ as desired.

Thus for normal $L$, the equivalence relation $\equiv$ above coincides with the one described in Lemma~\ref{lem:cokernel}, so its quotient is also described by that result. It follows that $q$ is a morphism of mosaics, where $M/L$ is equipped with the hyperoperation described in the statement of Lemma~\ref{lem:cokernel}. To see that $q$ is strict, we compute
\begin{align*}
q(x) \ostar q(y) &= (x \star L) \ostar (y \star L) \\ &= \{z \star L \mid z \in (x \star L) \star (y \star L) \} \\
&\subseteq \{z \star L \mid z \in (x \star y) \star L \} \\
&=\{z \star L \mid z \in x \star y\}\\
&= q(x \star y),
\end{align*}
where the penultimate equality holds because $\equiv$ above is an equivalence relation. Thus $q(x \star y) = q(x) \ostar q(y)$, showing that $q$ is a strict morphism.
The hyperoperation on $M/L$ can thus be written as
\[
(x \star L) \ostar (y \star L) = q(x) \ostar q(y) = q(x \star y) = \{z \star L \mid z \in x \star y\}
\]
as claimed.
\end{proof}

\begin{proposition}\label{prop:strict epi}
Let $R$ be an effective congruence on a (semi)mosaic $M$, and let $p \colon M \to M/R$ denote the canonical surjection. The following are equivalent:
\begin{enumerate}
\item The quotient morphism $p$ is strict.
\item For all $x_1, x_2, y_1, y_2 \in M$, if $x_i \equiv_R y_i$ for $i = 1,2$ then $(x_1 \star x_2) \equiv_R (y_1 \star y_2)$.
\end{enumerate}
If $M$ is a mosaic, then the above are further equivalent to:
\begin{enumerate}[resume]
\item $L = \ker p$ is normal in $M$, $R$ is the equivalence relation that partitions $M$ into cosets of $L$, and $p$ is isomorphic to the quotient $M \twoheadrightarrow M/L$ described in Lemma~\ref{lem:normal quotient}.
\end{enumerate}
\end{proposition}

\begin{proof}
(1)$\implies$(2): Assume $p$ is strict, and suppose $x_i, y_i \in M$ satisfy $x_i \equiv_R y_i$ for $i = 1,2$. Given $z_1 \in x_1 \star y_1$, we have 
\[
p(z_1) \in p(x_1 \star y_1) = p(x_1) \star p(y_1) = p(x_2) \star p(y_2) = p(x_2 \star y_2).
\]
This means that there exists $z_1' \in x_2 \star y_2$ such that $p(z_1) = p(z_1')$, which is equivalent to $z_1 \equiv_R z_1'$. Similar reasoning shows that for each $z_2 \in x_2 \star y_2$ there exists $z_2' \in x_1 \star y_1$ such that $z_2 \equiv_R z_2'$. So~(2) holds.

(2)$\implies$(1): Assume~(2) holds, and let $x,y \in M$. To show that $p$ is strict, it suffices to prove $p(p^{-1}(x) \star p^{-1}(y)) \subseteq p(x \star y)$. Let $z \in p^{-1}(x) \star p^{-1}(y)$; we will show $p(z) \in p(x \star y)$. This means there exist $x',y' \in M$ with $p(x) = p(x')$ and $p(y) = p(y')$ such that $z \in x' \star y'$. Because $p$ is the quotient map, this means $x \equiv_R x'$ and $y \equiv_R y'$. Since $x \star y \equiv_R x'\star y'$, there exists $w \in x \star y$ such that $z \equiv_R w$. This means that $p(z) = p(w) \in p(x \star y)$ as desired.

For the remainder of this proof, assume that $M$ is a mosaic. We have (3)$\implies$(1) by Lemma~\ref{lem:normal quotient}. Conversely, assume that~(1) holds. 
Let $L = \ker p$, which, it is easy to see, is a strict submosaic of $M$. We first prove that for $x,y \in M$,
\[
p(x) = p(y) \iff x \in y \star L.
\]
If $x \in y \star L$ then it is straightforward to check that $p(x) = p(y)$. Conversely, suppose that $p(x) = p(y)$. Then strictness of $p$ implies that
\[
e \in p(y)^{-1} \star p(x) = p(y^{-1} \star x).
\]
Thus there exists $z \in (y^{-1} \star x) \cap \ker p$. By reversibility, we have 
\[
x \in y \star z \subseteq y \star L
\]
as desired.

This means that the equivalence classes of $R$ are exactly the left cosets of $L$. But a symmetric argument also shows that these equivalence classes are equal to the right cosets as well. Thus for all $x \in M$ we have
\[
x \star L = [x]_R = L \star x,
\]
verifying condition~(i) of normality. To verify condition~(ii), fix $x,y \in M$. We have
\[
p((x \star L) \star (y \star L)) = p(x \star L) \star p(y \star L) = p(x) \star p(y) = p(x \star y).
\]
This implies that $(x \star L) \star (y \star L) \subseteq p^{-1}(p(x \star y)) = (x \star y) \star L$, proving~(ii). So $L$ is normal in $M$, and as explained in the proof of Lemma~\ref{lem:normal quotient}, the cokernel of the inclusion $L \hookrightarrow M$ is given by the coset quotient space $M/L$.
\end{proof}

For a general mosaic $M$ and strict submosaic $L \subseteq M$, the quotient $M \twoheadrightarrow M/L$ will not be strict if $L$ is not normal. (Any non-normal subgroup of a group yields an explicit example, such as $M=S_3$ and $L=\{e, (12)\}$). 
On the other hand, if $M$ is commutative and associative---that is, a canonical hypergroup---then any strict subhypergroup $L \subseteq M$ is automatically normal, so the quotient map $M \twoheadrightarrow M/L$ is strict. This was essentially observed in~\cite[Corollary~3.17]{Jun:geometry}.

\separate

Condition~(3) of
Proposition~\ref{prop:strict epi} characterizes strict epimorphisms whose domain is a \emph{mosaic} purely in terms of their kernel, but there is no similar characterization for \emph{semi}mosaics. 
Notice that every surjective monoid homomorphism is a strict epimorphism of semimosaics, but it is well known that such morphisms are not all cokernels. (For example, consider the Boolean monoid $\mathbb{B} = \{0,1\}$, where $1 + 1 = 1$, and the homomorphism $\mathbb{N}\rightarrow\mathbb{B}=\{0, 1\}$ given by $n\mapsto 1$ if $n>0$.) 
Thus we cannot characterize strict epimorphisms of semimosiacs purely in terms of their kernel.

\subsection{Quotients of some small groups}

We now describe a method to characterize all possible (semi)mosaic congruences by separating the kernel from the rest of the congruence.

\begin{theorem}\label{thm:pairs}
Let $M$ be a semimosaic. The semimosaic equivalences on $M$ are in bijective correspondence with pairs of the form $(L, \equiv)$ where $L \subseteq M$ is an absorptive subsemimosaic and $\equiv$ is an equivalence relation on $(M / L) \setminus \{[e]_L\}$, where $M/L$ is as described in Lemma~\ref{lem:cokernel}.

If $M$ is a mosaic, then such a pair corresponds to a mosaic equivalence if and only if and $\equiv$ preserves inverses (in the sense that $x \equiv y \implies x^{-1} \equiv y^{-1}$ for all $x,y \in M$).
\end{theorem}

\begin{proof}
Fix a regular epimorphism $q \colon M \twoheadrightarrow N$, and denote $L = \ker q$. Then $q$ factors through the cokernel of the inclusion $L \subseteq M$ as
\[
\begin{tikzcd}
M \ar[rr, twoheadrightarrow, "q"] \ar[dr, twoheadrightarrow, "p_1"] & & N. \\
& M/L \ar[ur, twoheadrightarrow, "p_2"] & 
\end{tikzcd}
\]
As $L = \ker q$, it must be the case that $\ker p_2$ is trivial.

Since $q = p_2 \circ p_1$ is a regular epimorphism, it follows directly that $p_2$ is still a regular epimorphism. Thus $p_2$ is isomorphic to a quotient $M/L \twoheadrightarrow (M/L)/\equiv_q$ by some semimosaic equivalence. Because $\ker q = L$, the equivalence under $\equiv_q$ class for identity $[e]_L$ is a singleton. This means that that $\equiv_q$ restricts to an equivalence relation on $(M/L) \setminus\{L\}$.

Conversely, consider a pair of an absorptive submosaic $L \subseteq M$ and an equivalence relation $\equiv$ on $M/L$ with trivial kernel. The quotients $p_1 \colon M \to M/L$ and $p_2 \colon M/L \to (M/L)/\equiv$ compose to a regular epimorphism $q = p_2 p_1$ as in the diagram above. This yields an inverse to the assignment $q \mapsto (\ker q, \equiv_q)$, establishing the bijection. 

The case where $M$ and $N$ are mosaics and $q$ is a morphism of mosaics corresponds to the case where $L = \ker q$ is a strict submosaic and $\equiv_q$ is a mosaic equivalence.
\end{proof}

If $M = G$ is a group, the characterization of equivalences above simplifies to the following. The submosaics $L \subseteq G$ are simply the subgroups of $G$. By Lemma~\ref{lem:cokernel}, the cokernel of the inclusion $L \hookrightarrow G$ in $\Msc$ is the double coset hypergroup $G \modmod L$ with hyperoperation
\[
(LxL) \ostar (LyL) = \{LzL \mid z \in xLy\}.
\]
This is isomorphic to the quotient group in case $L$ is normal. 

We close with a few examples of small groups to illustrate the number of mosaic equivalences $G$, corresponding to mosaic quotients of $G$ as in~\eqref{eq:quotient correspondence}. For each group $G$, it suffices by Theorem~\ref{thm:pairs} to find the pairs $(L, \equiv)$ where $L \leq G$ is a subgroup and $\equiv$ is an equivalence relation on $G\modmod L \setminus \{L\}$ that preserve inversion. If $L$ is a normal subgroup (e.g., if $G$ is abelian) then $G \modmod L = G/L$ is the ordinary quotient group.

\begin{example}
    Let $G=\mathbb{Z}/3\mathbb{Z}$, where the operation is written additively. 
    Note that $G$ has only two subgroups: $0$ or $G$. If $L=G$, then $G/L$ is the trivial mosaic and $\equiv$ is empty. Othewise $L=0$, and $\equiv$ is a relation on $\{\pm 1\}\subset G$. Thus $S$ is either the diagonal relation or the universal relation. The former gives the group $G=\mathbb{Z}/3\mathbb{Z}$ and the latter gives the Krasner hypergroup $\K$. In summary, the quotient mosaics of $\mathbb{Z}/3\mathbb{Z}$ are $\mathbb{Z}/3\mathbb{Z}$, $\K$ and $0$.
\end{example}

\begin{example}
    Let $G=\mathbb{Z}/5\mathbb{Z}$, with the operation is written additively again. 
    Note that $G$ has only two subgroups: $0$ or $G$. If $L=G$, then $G/L$ is the trivial mosaic and $\equiv$ is empty. Othewise $L=0$, and $\equiv$ is a relation on $T:=\{\pm 1, \pm2\}\subset G$. There are six possibilities for $S$:
\begin{enumerate}
    \item $\Delta_T$ 
    \item $\Delta_T\cup\{(1, -1), (-1, 1)\}$
    \item $\Delta_T\cup\{(2, -2), (-2, 2)\}$ 
    \item $\Delta_T\cup\{(1, 2), (2, 1), (-1, -2), (-2, -1)\}$
    \item $\Delta_T\cup\{(1, -2), (-2, 1), (-1, 2), (2, -1)\}$
    \item $T\times T$
\end{enumerate}
where $\Delta_T$ is the diagonal relation. The equivalence~$(1)$ gives the group $G$ and~$(6)$ gives the Krasner hypergroup $\K$. The folloing are lists of elements of $G/R$ and operation tables for the cases (2)--(5): 
\begin{itemize}
\item[(2)] $G/R=\{0, 1, \pm 2\}$ 
\begin{center}
\begin{tabular}{ |c|c|c|c|c| } 
\hline
+ & $0$ & $1$ & $2$ & $-2$ \\
\hline
$0$ & $0$ & $1$ & $2$ & $-2$ \\ 
\hline
$1$ & $1$ & $\{0, 2, -2\}$ & $\{1, -2\}$
& $\{1, 2\}$\\ 
\hline
$2$ & $2$ & $\{1, -2\}$ & $1$
& $0$\\ 
\hline
$-2$ & $-2$ & $\{1, 2\}$ & $0$
& $1$\\ 
\hline
\end{tabular}
\end{center}
Notice that $1+(2+2) \neq (1+2)+2$, so this is not a hypergroup.
\item[(3)] $G/R=\{0, \pm1, 2\}$ \begin{center}
\begin{tabular}{ |c|c|c|c|c| } 
\hline
+ & $0$ & $1$ & $-1$ & $2$ \\
\hline
$0$ & $0$ & $1$ & $-1$ & $2$ \\ 
\hline
$1$ & $1$ & $2$ & $0$
& $\{-1, 2\}$\\ 
\hline
$-1$ & $-1$ & $0$ & $2$
& $\{1, 2\}$\\ 
\hline
$2$ & $2$ & $\{-1, 2\}$ & $\{1, 2\}$
& $\{0, 1, -1\}$\\ 
\hline
\end{tabular}
\end{center}   
The quotients $G/R$ in~(2) and~(3) are isomorphic by the homomorphism $1\mapsto 2, 2\mapsto 1, -2\mapsto -1$.
\item[(4)] $G/R=\{0, \pm1\}$ 
\begin{center}
\begin{tabular}{ |c|c|c|c|c| } 
\hline
+ & $0$ & $1$ & $-1$ \\
\hline
$0$ & $0$ & $1$ & $-1$\\ 
\hline
$1$ & $1$ & $\{1, -1\}$ & $\{0, 1, -1\}$
\\ 
\hline
$-1$ & $-1$ & $\{0, 1, -1\}$ & $\{1, -1\}$
\\ 
\hline
\end{tabular}
\end{center} 
This is a hypergroup, isomorphic to $H_{3, 3}$ of~\cite[\S 7.1]{Zieschang}.
\item[(5)] $G/R=\{0, \pm1\}$ 
\begin{center}
\begin{tabular}{ |c|c|c|c|c| } 
\hline
+ & $0$ & $1$ & $-1$ \\
\hline
$0$ & $0$ & $1$ & $-1$\\ 
\hline
$1$ & $1$ & $\{1, -1\}$ & $\{0, 1, -1\}$
\\ 
\hline
$-1$ & $-1$ & $\{0, 1, -1\}$ & $\{-1, 1\}$
\\ 
\hline
\end{tabular}
\end{center}  
Once again, the quotients $G/R$ defined in~(4) and~(5) are isomorphic. 
\end{itemize}
In summary, the quotient mosaics of $\mathbb{Z}/5\mathbb{Z}$ are $\mathbb{Z}/5\mathbb{Z}, \{0, \pm 1, 2\}$ described in (3), $\{0, \pm1\}$ described in (4), $\K$ and $0$.
\end{example}

\begin{example}
    Let $G=S_3$. Then the subgroups of $G$ are
$$\{e\}, \{e, (12)\}, \{e, (13)\}, \{e, (23)\}, \{e, (123), (132)\}, S_3.$$
If $L=S_3$, then $G/L$ is the trivial group. If $L=\{e, (123), (132)\}$, then $G/L=G/H$ since $H\subset G$ is a normal subgroup.
Since $G/H\setminus\{[e]_H\}$ is a singleton, it has a unique relation, namely the trivial relation. This gives the cyclic group $\mathbb{Z}/2\mathbb{Z}$. 

Now let $L=\{e, (12)\}$. The corresponding double coset space 
$G \modmod L = \{L, L(13)L\}$ has order two, so that $G \modmod L \setminus \{L\}$ is a singleton with only the trivial congruence.
In this case $G \modmod L$ is isomorphic to the Krasner hypergroup $\K$. The symmetric argument shows that the quotient by either $L = \{e, (23)\}$ or $L = \{e, (13)\}$ gives the Krasner hypergroup.

Finally, let $L=\{e\}$. In this case we must consider inverse-preserving equivalence relations on $$T=G\setminus\{e\}=\{(12), (13), (23), (123), (132)\}.$$ 
Consider the possibilities of the (inverse-stable) quotient sets. We list them as follows:
\begin{enumerate}
    \item If no two distinct elements are identified, then the quotient set is $T$. This gives the group $G=S_3$.
    \item Identifying two of the three transpositions produces the following three quotient sets consisting of four elements:
\begin{gather*}
\{[(12), (13)], [(23)], [(123)], [(132)]\}, \quad
    \{[(12), (23)], [(13)], [(123)], [(132)]\}, \\
    \{[(23), (12)], [(23)], [(123)], [(132)]\}.
\end{gather*}
The first quotient gives the following:
\begin{center}
\begin{tabular}{ |c|c|c|c|c|c| } 
\hline
$\star$ & $e$ & $a=[12]$ & $b=[23]$ & $c=[123]$ & $d=[132]$ \\
\hline
$e$ & $e$ & $a$ & $b$ & $c$ & $d$ \\ 
\hline
$[12]$ & $a$ & $\{e, c, d\}$ & $\{c, d\}$
& $\{a, b\}$ & $\{a, b\}$\\ 
\hline
$[23]$ & $b$ & $\{c, d\}$ & $e$
& $a$ & $a$\\ 
\hline
$[123]$ & $c$ & $\{a, b\}$ & $a$
& $d$ & $e$\\ 
\hline
$[132]$ & $d$ & $\{a, b\}$ & $a$
& $e$ & $c$\\ 
\hline
\end{tabular}
\end{center} 
This is not a hypergroup: $(b\star c)\star d=a\star d=\{a, b\}$ whereas $b\star (c\star d)=b\star e=b$.

\item Identifying the three transpositions gives the following set with three elements:
$$\{[(12), (13), (23)], [(123)], [(132)]\}.$$  
\begin{center}
\begin{tabular}{ |c|c|c|c|c| } 
\hline
$\star$ & $e$ & $a=[12]$ & $b=[123]$ & $c=[132]$ \\
\hline
$e$ & $e$ & $a$ & $b$ & $c$ \\ 
\hline
$[12]$ & $a$ & $\{e, b, c\}$ & $a$
& $a$\\ 
\hline
$[123]$ & $b$ & $a$ & $c$
& $e$\\ 
\hline
$[132]$ & $c$ & $a$ & $e$
& $b$\\ 
\hline
\end{tabular}
\end{center}   
This is a hypergroup, isomorphic to $H_{4, 2}$ of~\cite[\S 7.2]{Zieschang}.

\item Identifying the two 3-cycles gives the following set with four elements:
$$\{[(12)], [(13)], [(23)], [(123), (132)]\}.$$  
\begin{center}
\begin{tabular}{ |c|c|c|c|c|c| } 
\hline
$\star$ & $e$ & $a=[12]$ & $b=[13]$ & $c=[23]$ & $d=[123]$ \\
\hline
$e$ & $e$ & $a$ & $b$ & $c$ & $d$ \\ 
\hline
$[12]$ & $a$ & $e$ & $d$
& $d$ & $\{b, c\}$\\ 
\hline
$[13]$ & $b$ & $d$ & $e$
& $d$ & $\{a, c\}$\\ 
\hline
$[23]$ & $c$ & $d$ & $d$
& $e$ & $\{a, b\}$\\ 
\hline
$[123]$ & $d$ & $\{b, c\}$ & $\{a, c\}$
& $\{a, b\}$ & $\{e, d\}$\\ 
\hline
\end{tabular}
\end{center} 
This is not a hypergroup: $(a\star a)\star b=b$ wheras
$a\star (a\star b)=a\star d=\{b, c\}$.

\item Identifying one of the three transpositions and the 3-cycles produces the following three quotient sets consisting of three elements:
\begin{gather*}
\{[(12), (123), (132)], [(13)], [(23)]\},
    \quad \{[(13), (123), (132)], [(12)], [(23)]\}, \\
    \{[(23), (123), (132)], [(12)], [(13)]\}.
\end{gather*}
Note that $(123)$ and $(132)$ must be in the same equivalence class since they are inverse to each other(for example, $(12)\sim (123)$ implies $(12)=(12)^{-1}\sim (123)^{-1}=(132)$). The first quotient set gives the following:
\begin{center}
\begin{tabular}{ |c|c|c|c|c| } 
\hline
$\star$ & $e$ & $a=[12]$ & $b=[13]$ & $c=[23]$ \\
\hline
$e$ & $e$ & $a$ & $b$ & $c$ \\ 
\hline
$[12]$ & $a$ & $\{e, b, c\}$ & $\{a, c\}$
& $\{a, b\}$\\ 
\hline
$[13]$ & $b$ & $\{a, c\}$ & $e$
& $a$\\ 
\hline
$[23]$ & $c$ & $\{a, b\}$ & $a$
& $e$\\ 
\hline
\end{tabular}
\end{center} 
This is not a hypergroup: $b\star (b\star c)=b\star a=\{a, c\}$ whereas $(b\star b)\star c=e\star c=c$.

\item Identifying two of the three transpositions and the 3-cycles produces the following three quotient sets consisting of two elements:
\begin{gather*}
\{[(12), (13), (123), (132)], [(23)]\},
    \quad \{[(12), (23), (123), (132)], [(13)]\}, \\
    \{[(13), (23), (123), (132)], [(12)]\}.
\end{gather*}
\begin{center}
\begin{tabular}{ |c|c|c|c|c| } 
\hline
$\star$ & $e$ & $a=[12]$ & $b=[23]$ \\
\hline
$e$ & $e$ & $a$ & $b$\\ 
\hline
$[12]$ & $a$ & $\{e, a, b\}$ & $a$
\\ 
\hline
$[23]$ & $b$ & $a$ & $e$
\\ 
\hline
\end{tabular}
\end{center}  
This structure is a hypergroup, isomorphic to $H_{3,8}$ of~\cite[\S 7.1]{Zieschang}.
\item Identifying all the elements gives the singleton ${[T]}$. In this case, the quotient is a hypergroup isomorphic to $\K$.
\end{enumerate}
\end{example}

We note in the above examples that if $G$ is a group, there are mosaic quotients $G/\!\!\equiv$ that are non-associative mosaics, and thus which are not hypergroups. The only property of a mosaic of the form $G/\!\!\equiv$ that we are currently able to deduce for certain is that its hyperoperation is total. Thus we end with the following question.

\begin{question}
    Which total mosaics can be represented as regular epimorphic images of a group? In particular, can every hypergroup be represented as a regular epimorphic image of a group?
\end{question}

\bibliographystyle{amsplain}
\bibliography{congruences-v1}

\end{document}